\documentclass[a4paper,10pt]{article}
\usepackage[latin1]{inputenc}
\usepackage[T1]{fontenc}
\usepackage[english]{babel}
\usepackage{amsmath,amscd,hyperref}   
\usepackage{amsfonts}
\usepackage{amssymb}
\usepackage{amsthm}
\usepackage{color}
\usepackage{enumerate}
 \usepackage{verbatim}
 
\newcommand{\md}{\mathrm{d}}
\newcommand{\lag}{\langle}
\newcommand{\rag}{\rangle}
\newcommand{\e}{\epsilon}
\newcommand{\Be}{\boldsymbol{E}}

\newcommand{\p}{\partial}
\newcommand{\dd}{{\textup d}}
\newcommand{\de}{\delta}
\newcommand{\ty}{\infty}
\newcommand{\la}{\lambda}
\newcommand{\bp}{\boldsymbol{\psi}}
\newcommand{\bz}{\boldsymbol{z}}
\newcommand{\bo}{\boldsymbol}

\renewcommand{\AA}{\mathcal{A}}
\newcommand{\BB}{\mathcal{B}}
\newcommand{\CC}{\mathcal{C}}
\newcommand{\DD}{\mathcal{D}}

\newcommand{\II}{\mathcal{I}}

\newcommand{\KK}{{\cal K}}
\newcommand{\LL}{\mathcal{L}}

\newcommand{\OO}{\mathcal{O}}
\newcommand{\PP}{\mathcal{P}}
\newcommand{\QQ}{\mathcal{Q}}

\renewcommand{\SS}{\mathcal{S}}

\newcommand{\UU}{\mathcal{U}}
\newcommand{\VV}{\mathcal{V}}

\newcommand{\R}{\mathbb{R}}
\newcommand{\C}{\mathbb{C}}
\newcommand{\N}{\mathbb{N}}
\newcommand{\Z}{\mathbb{Z}}
\newcommand{\Q}{\mathbb{Q}}

\numberwithin{equation}{section}

\DeclareRobustCommand*{\RefNeqsSecStab}{\ref{P:2.2}}
\DeclareRobustCommand*{\RefNeqsSecGenericite}{\ref{L:4.1}}

\theoremstyle{plain}
\newtheorem*{mth}{Main Theorem}

\newtheorem{theorem}{Theorem}[section]
\newtheorem{lemma}[theorem]{Lemma}
\newtheorem{proposition}[theorem]{Proposition}

\theoremstyle{definition}

\theoremstyle{remark}

\newtheorem{remark}[theorem]{Remark}

\title{Simultaneous global exact controllability of an arbitrary number of \textsc{1d} bilinear {S}chr\"odinger equations}
\author{
Morgan \textsc{Morancey}\footnote{CMLS UMR 7640, Ecole Polytechnique, 91128 Palaiseau, France; 
email: Morgan.Morancey@cmla.ens-cachan.fr},
Vahagn \textsc{Nersesyan}\footnote{ Laboratoire de Math\'ematiques, UMR CNRS 8100, Universit\'e de Versailles-Saint-Quentin-Yvelines, F-78035 Versailles, France;  e-mail: Vahagn.Nersesyan@math.uvsq.fr}
}
\date{}

\begin{document}

\maketitle
\begin{abstract}
We consider a system of an arbitrary number of \textsc{1d} linear Schr\"odinger equations on a bounded interval with bilinear control. We prove global exact controllability in large time of these $N$ equations with a single control. This result is valid for an arbitrary potential with generic assumptions on the dipole moment of the considered particle. Thus, even in the case of a single particle, this result extends the available literature. The proof combines local exact controllability around finite sums of eigenstates, proved with Coron's return method, a global approximate controllability property, proved with Lyapunov strategy, and a compactness argument.



\end{abstract}

\tableofcontents

\section{Introduction}

 The evolution of a \textsc{1d} quantum particle submitted to an external laser field   is described by the following linear Schr\"odinger equation
\begin{equation} \label{syst_1eq}
\left\{
\begin{aligned}
& i \partial_t \psi = \left(-\partial^2_{xx} + V(x)\right) \psi - u(t) \mu(x) \psi, &(t,x)& \in (0,T) \times (0,1),
\\
&\psi(t,0) = \psi(t,1) = 0,
\end{aligned}
\right.
\end{equation}
where $V(x)$ is the potential   of the particle,     $\mu(x)$ is the dipole moment,~$\psi(t,x)$ is the wave function, and~$u(t)$ is  the amplitude   of the  laser. In this setting, we consider $N$ identical and independent particles. Then neglecting entanglement effects, the system will be described by the following equations 
\begin{equation} \label{syst_Neq}
\left\{
\begin{aligned}
& i \partial_t \psi^j = \left(- \partial^2_{xx} + V(x)\right) \psi^j - u(t) \mu(x) \psi^j,  &(t,x)& \in (0,T) \times (0,1),
\\
& \psi^j(t,0) = \psi^j(t,1) = 0, \;  &j&   \in\{1, \dots, N\},
\\
& \psi^j(0,x) = \psi^j_0(x).
\end{aligned}
\right.
\end{equation}
This can be seen as a step towards more sophisticated and realistic models. From the point of view of controllability, this is a bilinear control system where the state is the $N$-tuple of wave functions $(\psi^1, \dots, \psi^N)$ and the control is the real-valued function $u$. The main result of this article is the global exact controllability of (\ref{syst_Neq}) for  an arbitrary number $N$ of particles,  arbitrary potential~$V$, and a generic dipole moment~$\mu$.

\smallskip 
Before stating our main result, let us introduce some notations. We denote by  $\SS$  the unit sphere in $L^2((0,1),\C)$ and $\boldsymbol \SS:= \SS^N$. Since the functions $V, \mu$ and the control $u$ are real-valued, for any initial condition~$\bp_0:=(\psi^1_0 ,\ldots,\psi^N_0)$ in $ \bo \SS$, the solution $\bp(t):=(\psi^1(t),\ldots,\psi^N(t))$ belongs to $\bo \SS$.   
We say that the vectors $\bp_0, \bp_f\in \boldsymbol \SS $ are unitarily equivalent, if there is a unitary operator $\UU$ in $L^2$ such that $\bp_f= \UU \bp_0$, i.e. $\psi_f^j=\UU \psi_0^j$ for all $j=1, \ldots,N$.
Finally, we define the operator $A_V$ by
\begin{equation*}
\DD(A_V) := H^2\cap H^1_0((0,1),\C), \quad A_V \varphi := \left( -\partial^2_{xx} + V(x) \right) \varphi
\end{equation*}
and, for $s>0$, we set~$H^s_{(V)} := \DD\big( A_V^{s/2} \big)$ and write $\boldsymbol H^{s}_{(V)}$ instead of~$(H^s_{(V)})^N$. 
\begin{mth}
For any given $V\in H^4((0,1),\R) $, problem (\ref{syst_Neq}) is globally exactly controllable in  $\boldsymbol H^{4}_{(V)}  $ generically with respect to $\mu$   in $ H^4((0,1),\R)$. More precisely, there is a residual set $\QQ_V$ in  $ H^4((0,1),\R)$ such that for any $\mu\in \QQ_V$ and for any unitarily equivalent vectors $\bp_0,\bp_f\in\boldsymbol \SS\cap \boldsymbol H^{4}_{(V)} $  there is a time $T>0$ and a control $u\in L^2((0,T),\R)$ such that the solution of (\ref{syst_Neq}) satisfies 
$$
\bp(T)= \bp_f.
$$
\end{mth}
First of all, notice that the unitary equivalence assumption on the initial condition and the target is not restrictive. Indeed, the evolution of the considered Schr\"odinger equation (\ref{syst_1eq}) is unitary, hence the system can be controlled from a given initial state only to a unitarily equivalent target.

\smallskip

 The problem of controllability for the bilinear Schr\"odinger equation   has been widely studied in the literature. 
A  negative controllability result for bilinear quantum systems is proved by Turinici \cite{Turinici00} as a corollary of a   general result by Ball, Marsden, and Slemrod \cite{BallMarsdenSlemrod82}. It states that the complement of the reachable set with $L^2$ controls from any initial condition in $\SS \cap H^2_{(0)}$ is dense in $\SS \cap H^2_{(0)}$. Thus, these equations have been considered to be non-controllable.

This negative result is actually only due to the choice of the functional setting. For a single particle, Beauchard proved in \cite{Beauchard05} local exact controllability in large time in $H^7_{(0)}$ in the case $\mu(x)=x$, $V(x)=0$, using Coron's return method, quasi-static deformations, and Nash--Moser theorem. Exhibiting a regularizing effect, this result was extended to the case of the space  $H^3_{(0)}$ for generic dipole moment $\mu$, still in the case $V=0$, by Beauchard and Laurent \cite{BeauchardLaurent}. Thus, as we are dealing with an arbitrary potential $V$ and a generic dipole moment $\mu$,   Main~Theorem with $N=1$ is already an improvement of the previous literature. In \cite{BeauchardCoron06}, Beauchard and Coron proved exact controllability between eigenstates for a particle in a moving potential well as studied by Rouchon in \cite{RouchonModele}.

\smallskip

Different methods have been developed to study approximate controllability.    A first  strategy of the proof of approximate controllability is due to Chambrion, Mason, Sigalotti, and Boscain~\cite{CMSB09}, which relies on the  geometric techniques based on the controllability of the Galerkin approximations.
The   hypotheses of this result were refined by Boscain, Caponigro, Chambrion, and Sigalotti in \cite{BCCS11}. In a more recent paper~\cite{BoscainCaponigroSigalotti13} of this team, in particular, it is proved   a  simultaneous approximate controllability property in Sobolev spaces for  an arbitrary number of equations. For more details and more references about the geometric techniques, we refer the reader  to the recent survey~\cite{BoscainChambrionSigalotti_review}. Although the results presented in these papers cover an important  class of models, the functional setting used there is always incompatible with the one which is necessary for the exact controllability. More precisely, approximate controllability is proved in less regular spaces than the one needed for exact controllability.

The second method which is used in the literature to prove approximate controllability for the bilinear Schr\"odinger equation is the Lyapunov strategy.  This method was   used by Mirrahimi in \cite{Mirrahimi09}    in the case of a mixed spectrum and by Beauchard and Mirrahimi in \cite{BeauchardMirrahimi09} in the case $V=0$ and $\mu(x)=x$. Both of these results prove approximate stabilization  in   $L^2$. Global approximate controllability   with generic assumptions both on the potential and the dipole moment is obtained by the second author in \cite{Nersesyan} and extended to higher norms leading to the first global exact controllability result for a bilinear quantum system in \cite{Nersesyan10}. For a model involving also a quadratic control, we refer to~\cite{Morancey_polarisabilite}. Approximate controllability in regular spaces (containing $H^3$) can also be deduced from the exact controllability results in infinite time \cite{NersesyanNersisyan1D,NersesyanNersisyan12} by Nersisyan and the second author.  The novelty of   Main Theorem with respect to the above papers is the fact that $N$~particles are controlled simultaneously in a regular space for an arbitrary fixed potential $V$.

Simultaneous exact controllability of quantum particles has been obtained for a finite dimensional model in \cite{TuriniciRabitz04} by Turinici and Rabitz. Their model uses specific orientation of the molecules and their proof relies on iterated Lie brackets.
To our best knowledge, the only exact simultaneous controllability results for infinite dimensional bilinear quantum systems were obtained in \cite{Morancey_simultane} by the first author locally around eigenstates in the case $V=0$ for $N=2$ or $N=3$. This is proved either up to a global phase in arbitrary time or exactly up to a global delay in the case $N=2$ and up to a global phase and a global delay in the case $N=3$. In that paper, it is also proved that, under generic assumptions on the dipole moment, local exact controllability (resp. local controllability up to a global phase) with controls small in $L^2$ does not hold in small time for $N\geq 2$ (resp. $N\geq 3$). A key issue for the positive results of this paper is the construction of a suitable reference trajectory which coincides (up to global phase and/or a global delay) at the final time with the vector of eigenstates. Extending directly this result to the case $N \geq 4$ presents two difficulties: in the trigonometric moment problem we solve for the construction of the reference trajectory resonant frequencies appear (e.g. $\lambda_7 - \lambda_1 = \lambda_8 - \lambda_4$) and the frequency~$0$ appears with multiplicity $N$. The use of a global phase and/or a global delay, by adding new degrees of freedom, allowed to deal with the frequency $0$ having multiplicity two or three. In our setting, we do not impose any conditions on the phase terms of the reference trajectory (see Proposition~\ref{prop_controle_traj_ref}). Thus, the frequency $0$ does not appear in the associated trigonometric moment problems. Taking advantage of the assumptions on the spectrum of the free operator, we  prove local exact controllability around $(\varphi_{1,V}, \dots, \varphi_{N,V})$ (see the First step of the proof of Theorem~\ref{th_controle_exact_local}). The price to pay is that we lose track of the time of control.


\paragraph*{Structure of the article.}
The Main Theorem is proved in three steps. First, under favourable hypotheses on $V$ and $\mu$, we prove that any initial condition can be driven arbitrarily close to some finite sum of eigenfunctions. This is done in Section~\ref{sect_controle_approche} using a Lyapunov strategy inspired by \cite{Nersesyan10}. Then, adapting the ideas of \cite{Morancey_simultane}, using favourable assumptions on the spectrum of $A_V$ and a compactness argument, we prove in Section~\ref{sect_control_exact} exact controllability locally around specific finite sums of eigenfunctions. Finally, for any potential $V$, using a perturbation argument, leading to the potential $V+\mu$ instead of $V$, we gather in Section~\ref{sect_control_global} the two previous results to prove the Main Theorem. Let us mention that, essentially with the same proof, one can prove global exact controllability in~$H^{3+\e}$, for any~$\e>0$.

\subsection*{Notations}

The space $L^2((0,1),\C)$ is endowed with the usual scalar product  
\begin{equation*}
\lag f,g \rag = \int_0^1 f(x) \overline{g(x)} \md x,
\end{equation*}
and we denote by $\| \cdot \|$ the associated norm. For any $s>0$, we denote by $\| \cdot \|_s$ the classical norm on the Sobolev space $H^s((0,1),\C)$.
The eigenvalues and eigenvectors of the operator $A_V$ are  denoted respectively by $\lambda_{k,V}$ and $\varphi_{k,V}$. The eigenstates are defined by 
\begin{equation*}
\Phi_{k,V}(t,x) := \varphi_{k,V}(x) e^{-i \lambda_{k,V} t}, \quad (t,x) \in \R^+ \times (0,1), \,  k \in \N^*.
\end{equation*}
Any $N$-tuple of eigenstates is a solution of system~(\ref{syst_Neq}) with control $u \equiv 0$.   
Notice that
$$
H^3_{(V)} = \left\{ \varphi\in H^3((0,1),\C)\, ; \, \varphi|_{x=0,1}=\varphi''|_{x=0,1}=0 \right\} =H^3_{(0)}
$$
for any $V\in H^3((0,1),\R)$.
We endow this space with the norm
$$
\| \psi \|_{H^3_{(V)}} := \left( \sum_{k=1}^{\infty} | k^3 \lag \psi , \varphi_{k,V} \rag |^2 \right)^{\frac{1}{2}}.
$$
We use bold characters to denote vector functions or product spaces.  For instance,   we denote by  $\bp(t)$ the vector $(\psi^1(t),\ldots,\psi^N(t))$  of  solutions of \eqref{syst_Neq} and  by $\boldsymbol H^s_{(V)}$ the space $(H^s_{(V)} )^N$. With coherent notations, $\bo \varphi_V$ denotes the vector  $(\varphi_{1,V},\dots,\varphi_{N,V})$.

\smallskip
\noindent
Let us denote by $U(H)$ the set of unitary operators from a Hilbert space $H$ into itself, and by $U_N$ the set of $N\times N$ unitary matrices.   Any $N\times M$ matrix~$C=(c_{ij})$ defines a linear map from $H^M$ to $H^N$ (denoted again by $C$) which associates to the vector   $(z^1, \ldots, z^M)$ the vector $( \sum_{j=1}^Mc_{1j}z^j, \ldots, \sum_{j=1}^Mc_{Nj}z^j)$.

\smallskip
\noindent
For a Banach space $X$, let    $B_X(a, d)$ be the closed ball of radius $d > 0$ centred on $a\in  X$. A subset of  $  X$ is said to be residual if   it contains a countable intersection of open and dense sets.

\noindent
The symbol $\delta_{j=k}$ is the classical Kronecker symbol, i.e., $\delta_{j=k} = 1$ if $j=k$ and $\delta_{j=k} = 0$ otherwise.

\smallskip
\noindent
Finally, we  define the space
\begin{equation*}
\ell^2_r (\N,\C) := \left\{ d \in \ell^2(\N,\C) \, ;\, d_0 \in \R \right\}
\end{equation*}which is endowed with the natural metric.

\section{Well-posedness}
\label{subsect_bien_pose}

In the following proposition,   we recall a well-posedness result   of the Cauchy problem for the Schr\"odinger equation
\begin{equation} \label{1eq_cauchy}
\left\{
\begin{aligned}
& i \partial_t \psi = \left(- \partial^2_{xx} + V(x)\right) \psi- u(t) \mu(x) \psi - v(t) \mu (x) \zeta,  &(t,x)& \in (0,T) \times (0,1),  \\
& \psi(t,0) = \psi(t,1) = 0,  
\\
& \psi(0,x) = \psi_0(x),
\end{aligned}
\right.
\end{equation}
and list properties of the solution that will be used in the proofs of the main results in the subsequent sections. 

\begin{proposition} \label{P:2.1}
Let us assume that $V, \mu \in H^3((0,1),\R)$ and $T>0$.  Then, for any $\psi_0\in H^3_{(0)}$, $\zeta \in C^0([0,T],H^3_{(0)})$ and  $u,v \in L^2((0,T),\R)$ there is a unique weak solution of (\ref{1eq_cauchy}), i.e.,  a function $\psi \in C( [0, T ],H^3_{(0)} )$ such that the following equality holds in $H^3_{(0)}$  for every $t \in [0, T ]$ 
$$
\psi(t)=e^{-iA_Vt}\psi_0+i\int_0^te^{-iA_V(t-\tau)} \big( u(\tau) \mu \psi(\tau) + v(\tau) \mu \zeta(\tau) \big) \dd \tau.
$$
For every $R>0$, there exists $C=C(T,V,\mu,R)>0$ such that, if $\|u\|_{L^2(0,T)} < R$, this weak solution satisfies
$$
\| \psi \|_{C^0([0,T],H^3_{(V)})} \leq 
C \left( \| \psi_0 \|_{H^3_{(V)}} + \|v\|_{L^2(0,T)} \| \zeta \|_{L^{\infty}((0,T),H^3_{(V)})} \right).
$$
Moreover, if $v \equiv 0$  the solution satisfies 
$$
\|\psi(t)\| =\|\psi_0\|  \quad \text{for all $t\in [0,T]$},
$$ 
and the following properties hold in the case $v \equiv 0$.

\medskip

{\bf Differentiability.} Let us denote by $\psi(t, \psi_0,u)$ the solution of (\ref{1eq_cauchy}) corresponding to  $\psi_0\in H^3_{(0)}$,  $u\in L^2((0,T),\R)$ and $v=0$. The mapping
\begin{equation} \label{def_propagateur}
\begin{array}{cccc}
 \psi(T, \psi_0,\cdot) : &  L^2((0,T),\R)
& \rightarrow & H^3_{(0)}  ,
\\
& u & \mapsto & \psi(T, \psi_0,u)
\end{array}
\end{equation}
 is $C^1$, and for any $u,v\in L^2((0,T),\R)$, we have $\partial_{u} \psi(T, \psi_0,u)v=\Psi(T) $, where $\Psi$ is the weak solution of the linearized system
\begin{equation*}  
\left\{
\begin{aligned}
& i \partial_t \Psi = \left(- \partial^2_{xx} + V(x) \right) \Psi- u(t) \mu(x) \Psi-v(t)\mu(x)\psi,  \,\,\,\,  (t,x)  \in (0,T) \times (0,1),  \\
& \Psi(t,0) = \Psi(t,1) = 0,  
\\
& \Psi(0,x) = 0,
\end{aligned}
\right.
\end{equation*}
with $\psi = \psi( \cdot, \psi_0,u)$.

\smallskip

{\bf Regularity.} Assume that   $V, \mu \in H^4((0,1),\R)$.  For any   $ u\in W^{1,1}((0,T),\R)
 $   and        $ \psi_0\in H^4_{(V- u(0)\mu) } $, we have 
    $\psi (t)\in H^4_{(V- u(t)\mu)}$  for all $t\in [0,T]$.

\smallskip

{\bf Time reversibility.}
Suppose that $\psi(T, \overline{\psi_f},u)=\overline{\psi_0}$ for some $\psi_0,\psi_f \in   H^3_{(0)}$, $u\in
L^2((0,T), \R)$, and $T>0$. Then $\psi(T,  {\psi_0},w)= {\psi_f}$, where $w(t)=u(T-t)$.

\end{proposition}

See \cite[Propositions~2 and~3]{BeauchardLaurent} for the proof of  the well-posedness in $H_{(0)}^3$  and for the differentiability property. The property of regularity  is established in~\cite[Proposition~47]{Beauchard05}. In these references,    the case of $V=0$ is considered,  but   the case of a non-zero~$V$ is proved by literally the same arguments (see~\cite{NersesyanNersisyan1D}).  The time reversibility property is obvious. 
 Proposition~\ref{P:2.1} implies that similar properties hold for the solutions   of system \eqref{syst_Neq}. We denote by $\bp(t, \bp_0,u)$ the solution of  \eqref{syst_Neq} corresponding to $\bp_0\in \boldsymbol H^3_{(0)}$ and $u\in L^2((0,T),\R)$.

\section{Approximate controllability  }
\label{sect_controle_approche}

\subsection[Approximate controllability towards finite sums of eigenvectors]{Approximate controllability towards finite sums of \\ eigenvectors}

In this section, we assume that the following conditions are satisfied for the  functions $V, \mu \in H^4((0,1),\R)$
\begin{description}
\item[$  \boldsymbol{\mathrm{(C_1)}}$] 
 $\lag \mu \varphi_{j,V}, \varphi_{k,V} \rag \neq 0$  for all $j \in \{1,\dots,N\}$, $k \in \N^*$.
 \item[$  \boldsymbol{\mathrm{(C_2)}}$] 
$\lambda_{j,V} - \lambda_{k,V} \neq \lambda_{p,V} - \lambda_{q,V}$ for all    $j \in \{1,\dots,N\}$, $k,p,q \in \N^*$ such that $\{j,k\} \neq \{ p,q\}$ and $k \neq j$.
\end{description}
 
For any   $M\in \N^*$, let us define the sets  
\begin{align}
\CC_M &:=\textup{Span}\{\varphi_{1,V}, \ldots, \varphi_{M,V}\}, \quad  \boldsymbol \CC_M :=(\CC_M)^N       ,\label{E1}
\\ 
\boldsymbol \Be &:=\left\{ \bp\in \boldsymbol L^2 \, ; \, \prod_{j=1}^N \lag \psi^j,\varphi_{j}\rag \neq 0   \right\}.\label{E2}
\end{align}

The following theorem is the main result of this section.
\begin{theorem}\label{Approx_contr}
Assume that Conditions $\mathrm{(C_1)}$ and $\mathrm{(C_2)}$ are satisfied for the functions $V,\mu\in H^4((0,1),\R)$. Then, for any    
  $\bp_0\in  \boldsymbol \SS\cap \boldsymbol H_{(V)}^4\cap \Be$,      
 there are  $M  \in \N^*$,  $\bp_f\in \boldsymbol \CC_M$,  sequences  $T_n>0$ and   $u_n\in C_0^\ty((0,T_n), \R)$   such that
\begin{equation}\label{E:lim}
\bp(T_n, \bp_0, u_n) \underset{n \to \infty}{\longrightarrow} \bp_f \quad\text{in}\quad  \boldsymbol H^{3}. 
\end{equation}
\end{theorem}
\begin{proof} 
See~\cite[Theorem~2.3]{Nersesyan10} for the proof of a similar result in the case $N=1$ (in that case one gets $M=1$).   To simplify notations, we shall write $\la_k, \varphi_k$ instead of $\la_{k,V}, \varphi_{k,V}$. 
     For any $\boldsymbol z=(z^1,\ldots, z^N)\in  \boldsymbol H^4_{(V)}$, let  us define the following   Lyapunov function
\begin{equation}\label{2.2}
\VV(\bz)= \alpha \sum_{j=1}^N \|( -\p^2_{xx}+{V} )^{ 2} \PP_{N}z^j\| ^2 +1- \prod_{j=1}^N|\lag z^j,\varphi_{j}\rag |^2,
\end{equation}where 
 $\alpha>0$ is a constant that will be chosen later
  and   $\PP_{N} $ is   the
orthogonal projection  in $L^2$ onto the closure of the vector span
of $\{\varphi_{k}\}_{k\ge N+1}$, i.e.,
\begin{equation} \label{def_proj}
\PP_N(z) := \sum_{k\ge N+1} \lag z, \varphi_k \rag \varphi_k.
\end{equation}
 Clearly, we have that $\VV(\bz )\ge 0$ for any $\bz\in \boldsymbol\SS\cap \boldsymbol H_{(V)}^4$ and  $\VV(\bz )= 0$   if and only if  $\bz=(c_1\varphi_1, \ldots, c_N \varphi_N)$ for some   $c_i\in \C$ such that $  |c_i|=1, i=1, \ldots, N$. Furthermore, 
 for any $\bz\in \boldsymbol\SS\cap \boldsymbol H_{(V)}^4$,  we have
\begin{align*}
\VV(\bz)\ge\alpha \sum_{j=1}^N \|( -\p^2_{xx}+{V} )^{ 2}\PP_{N}z^j\| ^2 \ge
C_1\sum_{j=1}^N \| z^j\| ^2_4-C_2 .
\end{align*}Thus
\begin{equation}\label{2.3}
C(1+\VV(\bz))\ge \|\bz\|_4^2
\end{equation}for some constant $C>0$.
We need the following result  which a generalization of~\cite[Proposition~2.6]{Nersesyan10}.  \begin{proposition}\label{P:2.2} Under the conditions of Theorem~\ref{Approx_contr},  for any $\bp_0\in  \boldsymbol \SS\cap \boldsymbol H_{(V)}^4\cap \Be \backslash \left(\cup_{M=1}^\ty\boldsymbol \CC_M\right)$   
  there is   a  time $T>0$  and a control~$u\in
C_0^\ty((0,T),\R)$  such that 
\begin{equation*}
\VV(\bp(T,\bp_0,u))<\VV(\bp_0).
\end{equation*}\end{proposition} 
See Section~\ref{S:2.3} for the proof of this result.

\smallskip
\noindent
Let us choose $\alpha>0$ in (\ref{2.2}) so small that   $\VV(\bp_0)<1$ and define the set
\begin{align*}
\KK:=\Big\{\bp\in  \boldsymbol H^{4}_{(V)} \, ; \,   \bp(T_n, \bp_0,u_{n}) \underset{n \to \infty}{\longrightarrow} \bp
\,\,\,\text{in } \boldsymbol H^{3}
  \,\,&\text{for some}\,\,  T_n\ge0,
\,\,\, \\&\,\, u_n \in C^\ty_0( (0,T_n),\R)
      \Big\}.\nonumber
\end{align*}
Then the infimum 
$
m:=\inf_{\bp\in\KK}\VV(\bp)
$ is attained, there is $\boldsymbol e\in \KK$ such that
\begin{equation}\label{2.4}
\VV(\boldsymbol e)=\inf_{\bp\in\KK}\VV(\bp).
\end{equation}
Indeed,   any minimizing sequence $\bp_n\in\KK$, $\VV(\bp_n)\rightarrow
m$ is bounded in $\boldsymbol H^4$, by~\eqref{2.3}. Extracting a subsequence if necessary, we may assume that    $\bp_n \rightharpoonup \boldsymbol e$ in~$ \boldsymbol H^{4}$ for some $\boldsymbol e\in \boldsymbol H^4_{(V)}$. This implies that
$\VV(\boldsymbol e)\le\liminf_{n\rightarrow\ty}\VV(\bp_n)= m$. Let us show that
$\boldsymbol e\in \KK$. As $\bp_n\in\KK$, there are sequences $T_n>0$ and $u_n\in
C_0^\ty((0,T_n),\R)$ such that
\begin{align}
\|\bp(T_n, \bp_0,u_n )-\bp_n\|_{H^3_{(V)}} 
&\le\frac{1}{n}.\label{2.5}
\end{align}
On the other hand,   $\bp_n\rightarrow \boldsymbol e$ in $\boldsymbol H^{3}$, and  (\ref{2.5}) implies that
$\bp(T_n,\bp_0,u_{n}) \rightarrow \boldsymbol e$ in $\boldsymbol H^{3}$. Thus
$\boldsymbol e\in\KK$ and  $\VV( \boldsymbol e)= m$.

Let us prove that $\boldsymbol e\in \boldsymbol\CC_M$ for some $M\in \N^*$. Suppose, by contradiction, that
$ \boldsymbol e \notin \cup_{M=1}^\ty\boldsymbol \CC_M   $. It follows from (\ref{2.4}) and from the choice of
$\alpha$ that
 $\VV(\boldsymbol e)\le\VV(\bp_0)<1$. This shows that $ \boldsymbol e \in \Be$.
 Proposition~\ref{P:2.2} implies that
       there are $T>0$ and 
$u\in C^\ty_0((0,T),\R)$  such that
\begin{equation}\label{2.6}
\VV(\bp(T, \boldsymbol e,u))<\VV(\boldsymbol e).
\end{equation}
Define $\tilde{u}_n(t)=u_n(t)$, $t\in[0,T_n]$ and
$\tilde{u}_n(t)=u (t-T_n)$, $t \in[T_n,T_n+T]$.  Then $\tilde{u}_n\in
C^\ty_0((0,T_n+T),\R)$ and, by the continuity in $H^{3}$ of the resolving operator for (\ref{syst_Neq}), we get 
$$\bp (T_n+T   , \bp_0,\tilde{u}_n)\rightarrow
\bp (T , \boldsymbol e,u)\,\,\,\text{in $\boldsymbol H^{3},$}
$$ hence 
$\bp(T, \boldsymbol e,u  )\in \KK$. Together with  (\ref{2.6}), this contradicts (\ref{2.4}). Thus $\boldsymbol e\in \boldsymbol \CC_M$, and we get \eqref{E:lim} with $\bp_f=\boldsymbol e.$

\end{proof}

\subsection{Proof of Proposition~\RefNeqsSecStab}\label{S:2.3}

Let  us take any vector $\bp_0\in  \boldsymbol \SS\cap \boldsymbol H_{(V)}^4 \cap \Be \backslash (\cup_{M=1}^{\infty} \boldsymbol \CC_M)$, any time $T>0$, any control $w\in C^\ty_0((0,T),\R)$, and 
consider the mapping
\begin{equation*}
\begin{array}{cccc}
\VV(\bp (T,\bp_0,(\cdot)w)) : & \R
& \rightarrow & \R , 
\\
& \sigma & \mapsto & \VV(\bp (T,\bp_0,\sigma w)).
\end{array}
\end{equation*}
It suffices to  show
that,   for an appropriate choice of $T$ and $w$, we have
\begin{equation}\label{2.7}
\frac{\dd \VV(\bp (T,\bp_0,\sigma w))}{\dd \sigma}\Big|_{\sigma=0}\neq
0.
\end{equation} Indeed, \eqref{2.7}  implies that there is
$\sigma_0\in\R$ close to zero such that
$$
\VV(\bp (T,\bp_0,\sigma_0 w))<\VV(\bp (T,\bp_0,0))=\VV( \bp_0 ),
$$which completes the proof of Proposition~\ref{P:2.2}.

\smallskip
To prove \eqref{2.7}, notice that 
\begin{align}\label{2.8}
&\frac{\dd \VV(\bp (T,\bp_0,\sigma w))}{\dd \sigma}\Big|_{\sigma=0}
\nonumber\\
&=2 \sum_{j=1}^N\Re\Big( \alpha\lag ( -\p^2_{xx}+{V} )^{ 2} \PP_{N} \psi^j(T), ( -\p^2_{xx}+{V} )^{ 2} \PP_{N} \Psi^j(T) \rag \nonumber 
\\
& \quad - \lag \psi^j(T), \varphi_{j} \rag  \lag \varphi_{j},\Psi^j(T)\rag\prod_{ q=1, q\neq j}^N|\lag \psi^q_0,\varphi_{q}\rag |^2\Big),
\end{align}
where
\begin{equation}\label{2.9}
\psi^j(t) =\psi(t,\psi_0^j,0)=\sum_{k=1}^\ty e^{-i\la_{k}t}\lag \psi_0^j,
\varphi_{k}\rag \varphi_{k},
\end{equation} 
  and $\Psi^j$ is the solution of the linearized problem 
\begin{equation*}  
\left\{
\begin{aligned}
& i \partial_t \Psi^j = \big(- \partial^2_{xx} + V(x)\big) \Psi ^j -w(t)\mu(x)\psi^j,   \, \,\, (t,x)  \in (0,T) \times (0,1),  \\
& \Psi^j(t,0) = \Psi^j(t,1) = 0,  
\\
& \Psi^j(0,x) = 0.
\end{aligned}
\right.
\end{equation*}  Rewriting this   in the Duhamel form 
$$
\Psi^j(t)= i\int_0^t e^{-iA_V(t-\tau)}w(\tau) \mu(x)\psi(\tau)\dd \tau
$$and using \eqref{2.9}, we get that 
\begin{equation}\label{2.10}
\lag \Psi^j(T), \varphi_{p}\rag= ie^{-i\la_{p}T}\sum_{k=1}^\ty\lag
\psi_0^j, \varphi_{k}\rag \lag \mu \varphi_{k},\varphi_{p}\rag\int_0^T
 e^{-i(\la_{k}-\la_{p})\tau} w(\tau)  \dd \tau.
\end{equation}Replacing (\ref{2.9}) and (\ref{2.10}) into
(\ref{2.8}), we obtain 
\begin{align*}
\frac{\dd \VV(\bp (T,\bp_0,\sigma w))}{\dd \sigma}\Big|_{\sigma=0}= \int_0^T\Phi(\tau)w(\tau)\dd\tau,
\end{align*}where
\begin{align}\label{2.11}
i\Phi(\tau)&:= \sum_{j=1}^N\bigg(  \sum_{p=N+1,k=1}^\ty \alpha   \la_{p}^4 \lag \psi_0^j,
\varphi_{p}\rag \lag \varphi_{k},\psi_0^j \rag \lag \mu \varphi_{k},\varphi_{p}\rag
 e^{ i(\la_{k}-\la_{p})\tau}  \nonumber
\\
&\quad-\sum_{p=N+1,k=1}^\ty \alpha   \la_{p}^4 \lag  \varphi_{p}, \psi_0^j\rag \lag  \psi_0^j, \varphi_{k} \rag \lag \mu \varphi_{k},\varphi_{p}\rag  e^{ -i(\la_{k}-\la_{p})\tau}  \nonumber
\\
&\quad-\Big(\prod_{ q=1, q\neq j}^N|\lag \psi^q_0,\varphi_{q}\rag |^2 \Big) \sum_{ k=1}^\ty  \lag \psi_0^j,\varphi_{j} \rag \lag \varphi_{k},\psi^j_ 0 \rag \lag \mu \varphi_{k},\varphi_{j}\rag 
 e^{ i(\la_{k}-\la_{j})\tau}     \nonumber
\\
&\quad+\Big(\prod_{ q=1, q\neq j}^N|\lag \psi^q_0,\varphi_{q}\rag |^2\Big) \sum_{ k=1}^\ty  \lag \varphi_{j},\psi^j_ 0 \rag \lag \psi_0^j, \varphi_{k}\rag \lag \mu \varphi_{k},\varphi_{j}\rag  e^{- i(\la_{k}-\la_{j})\tau} \bigg)     \nonumber
\\
& =: \sum _{1\le k<p<\ty} \left( P( k, p) e^{i(\la_{k} -\la_{p} )\tau} + \tilde{P}(k,p) e^{-i(\lambda_k - \lambda_p) \tau} \right),
\end{align}
where   $P(k,p)$ and $\tilde{P}(k,p)$ are constants. To prove (\ref{2.7}), it suffices to show that $\Phi(\tau)\neq0$ for some $\tau\ge0$. Suppose, by contradiction, that $\Phi(\tau)=0$ for all $\tau\ge0$. Then Condition $\mathrm{(C_2)}$ and~\cite[Lemma~3.10]{Nersesyan} imply that $P(k,p) =0$ for all $k<p$.
 Using the equality  $P(k,p) =0$ for $1\le k\le N<p<\ty$   and  $\mathrm{(C_1)}$, we get  that 
  $$
\left( \alpha \la_{p}^4 +  \prod_{ q=1, q\neq k}^N|\lag \psi^q_0,\varphi_{q}\rag |^2  \right) \lag \psi_0^k, \varphi_{p}\rag \lag \varphi_{k},\psi_0^k \rag+ \sum_{j=1, j\neq k}^N \alpha \la_{p}^4  \lag\psi_0^j, \varphi_{p}\rag \lag \varphi_{k},\psi_0^j \rag=0.
  $$ 
Assume that for some integer $p>N $ we have 
\begin{equation}\label{2.12}
\sum_{j=1}^N |  \langle \psi_{0}^j ,\varphi_{p}\rangle|>0. 
\end{equation}Let us set  $a_{k}(\la):=  \la +  \prod_{ q=1, q\neq k}^N|\lag \psi^q_0,\varphi_{q}\rag |^2   $ and  consider the determinant
$$ \Lambda(\la) =
 \left|
\begin{matrix}
      a_1(\la) \langle \psi_{0}^1 ,\varphi_{1} \rag & \la \langle \psi_{0}^2 ,\varphi_{1} \rag & \cdots & \la \langle \psi_{0}^N ,\varphi_{1} \rag \\
  \la  \langle \psi_{0}^1 ,\varphi_{2} \rag & a_2(\la) \langle \psi_{0}^2 ,\varphi_{2} \rag& \cdots & \la\langle \psi_{0}^N ,\varphi_{2} \rag \\
    \vdots               & \vdots               &    \ddots    & \vdots               \\
  \la  \langle \psi_{0}^1 ,\varphi_{N} \rag&\la \langle \psi_{0}^2 ,\varphi_{N} \rag & \cdots & a_N(\la) \langle \psi_{0}^N ,\varphi_{N} \rag
\end{matrix} \right|. $$
Then $\Lambda(\la) $ is a polynomial of degree less or equal to $N$ which vanishes at $\la=\alpha \la_{p}^4$. The free term in  $\Lambda(\la) $ is $\prod_{k=1}^N a_k(0) \lag \psi^k_0, \varphi_k \rag$ which is non-zero by the assumption $\bp_0 \in \bo E$. Thus $\Lambda(\la)$ has at most $N$ roots and the number of indices~$p$ such that (\ref{2.12}) holds is finite. This gives the existence of $M \in \N^*$ such that $\bp_0 \in \bo \CC_M$ and completes the proof of Proposition~\ref{P:2.2}.

\smallskip \hfill $\square$

\section{Local exact controllability}
\label{sect_control_exact}

\subsection{Local exact controllability around finite sums of eigenstates}
\label{subsect_controle_time_reversibility}

In this section, we assume that the following conditions are satisfied for the functions $V,\mu \in H^3((0,1),\R)$.
\begin{description}
\item[$  \boldsymbol{\mathrm{(C_3)}}$] There exists $C>0$ such that
\begin{equation*}
|\lag \mu \varphi_{j,V}, \varphi_{k,V} \rag| \geq \frac{C}{k^3}, \quad \forall j \in \{1,\dots,N\}, \, \forall k \in \N^*.
\end{equation*}
\item[$  \boldsymbol{\mathrm{(C_4)}}$] $\lambda_{k,V} - \lambda_{j,V} \neq \lambda_{p,V} - \lambda_{n,V}$ for all $j,n \in \{1, \dots,N\}$, $k \geq j+1$, $p \geq n+1$ with $\{j,k \} \neq \{p,n\}$.
\item[$  \boldsymbol{\mathrm{(C_5)}}$] $1, \lambda_{1,V}, \dots, \lambda_{N,V}$ are rationally independent.
\end{description}

The goal of this section is the proof of the following theorem.
\begin{theorem} \label{th_controle_exact_local}
Assume that Conditions $\mathrm{(C_3)}-\mathrm{(C_5)}$ are satisfied for $V, \mu \in H^3((0,1),\mathbb{R})$. Let  us take any $C_0, C_f \in U_N$   and   set $\bo z_0 := C_0 \bo \varphi_V$, $\bo z_f := C_f \bo \varphi_V$. Then there exist $\delta>0$ and $T>0$ such that if we define
\begin{align*}
\mathcal{O}_{\delta,C_0} &:=  \Big\{ \bo \phi \in {\bo H}^3_{(0)} \,; \,
\langle \phi^j,\phi^k \rangle = \delta_{j=k} \text{ and } \sum_{j=1}^N \| \phi^j - z_0^j\|_{H^3_{(V)}} < \delta \Big\},
\\
\mathcal{O}_{\delta,C_f} &:=  \Big\{ \bo \phi \in {\bo H}^3_{(0)} \,; \,
\langle \phi^j,\phi^k \rangle = \delta_{j=k} \text{ and } \sum_{j=1}^N \| \phi^j - z_f^j\|_{H^3_{(V)}} < \delta \Big\},
\end{align*}
then for any $\bo \psi_0 \in \OO_{\delta,C_0}$ and  $\bo \psi_f \in \mathcal{O}_{\delta,C_f}$, there is a control $u \in L^2((0,T),\mathbb{R})$ such that the associated solution of (\ref{syst_Neq}) with initial condition 
$\bo \psi(0) = \bo \psi_0$ satisfies $\bo \psi(T) = \bo \psi_f$.
\end{theorem}

\begin{remark} 
Notice that the condition
\begin{equation*}
\lag \phi^j ,\phi^k \rag = \delta_{j=k}, \quad \forall j,k \in \{1,\dots,N\}
\end{equation*}
is equivalent to the fact that $\bo \phi$ is unitarily equivalent to $\bo \varphi_V$. In this section, we will always consider such initial conditions. Thus, the associated trajectories will satisfy the following invariants
\begin{equation} \label{invariants}
\lag \psi^j (t) , \psi^k (t)  \rag \equiv \delta_{j=k}, \quad \forall j,k \in \{1,\dots,N\}.
\end{equation}
\end{remark}

\begin{remark} \label{rmk_quantum_gate}
A quantum logical gate is a unitary operator $\hat{\UU}$ in $L^2((0,1),\C)$ such that for some $n \in \N^*$, the space $\text{Span} \{ \varphi_{1,V}, \dots, \varphi_{n,V} \}$ is stable for $\hat{\UU}$. 
Designing such a quantum gate means finding a control $u \in L^2((0,T),\R)$ such that the associated solution of (\ref{syst_Neq}) with initial condition $(\varphi_{1,V}, \dots, \varphi_{n,V})$ satisfies
\begin{equation*}
\left( \psi^1(T), \dots , \psi^n(T) \right) = \left( \hat{\UU} \varphi_{1,V}, \dots, \hat{\UU} \varphi_{n,V} \right).
\end{equation*}
See \cite{BCClogical_gate} for $L^2$--approximate realization of such quantum logical gates with error estimates and numerical simulations on two classical examples.
Theorem~\ref{th_controle_exact_local} thus proves exact realization of quantum logical gates in large time under Conditions~$\mathrm{(C_3)}-\mathrm{(C_5)}$ of size $n$. Applying directly our Main Theorem leads to exact realization of any quantum gate, for an arbitrary potential with a generic dipole moment. 
\end{remark}

The proof of  Theorem~\ref{th_controle_exact_local}  is based on the following   proposition which is an adaptation of~\cite[Theorem~1.5]{Morancey_simultane}.
\begin{proposition} \label{prop_controle_traj_ref}
Assume that Conditions $\mathrm{(C_3)}$ and  $\mathrm{(C_4)}$ are satisfied for $V, \mu \in H^3((0,1),\R)$. For any $T>0$, there exist $\theta_1, \dots, \theta_N \in \R$, $\delta >0$, and a~$C^1$ map
\begin{equation*}
\Gamma : \mathcal{O}_\delta^0 \times \mathcal{O}_\delta^f \to L^2((0,T),\R),
\end{equation*}
where
\begin{align*}
\mathcal{O}_\delta^0 :&=  \Big\{  \bo \phi \in {\bo H}^3_{(0)} \,; \,
\lag \phi^j,\phi^k \rag = \delta_{j=k} \text{ and } \sum_{j=1}^N \| \phi^j - \varphi_{j,V}\|_{H^3_{(V)}} < \delta \Big\},
\\
\mathcal{O}_\delta^f :&=  \Big\{  \bo \phi \in {\bo H}^3_{(0)} \,; \,
\lag \phi^j,\phi^k \rag = \delta_{j=k} \text{ and } \sum_{j=1}^N \| \phi^j -  e^{i \theta_j}\varphi_{j,V}\|_{H^3_{(V)}} < \delta \Big\},
\end{align*}
such that for any initial condition $\bo \psi_0 \in \mathcal{O}_\delta^0$ and for any target $\bo \psi_f \in \mathcal{O}_\delta^f$, the solution of system (\ref{syst_Neq}) associated to the control $u := \Gamma \big( \bo \psi_0, \bo \psi_f \big)$ satisfies $\bo \psi(T) = \bo \psi_f$.
\end{proposition}
In the case $N=2$ and $V=0$, the previous proposition is exactly \cite[Theorem~1.2]{Morancey_simultane} with $\theta_j = \theta - \lambda_{j,V} T$. As here we do not impose any condition on the phase terms $\theta_j$, the proof of Proposition~\ref{prop_controle_traj_ref} does not introduce new ideas with respect to \cite{Morancey_simultane}. Anyway, dealing with an arbitrary number of equations (instead of two or three equations in \cite{Morancey_simultane}) needs some adaptations that are described in Sections~\ref{subsect_traj_ref}, \ref{subsect_control_linearise}, and \ref{subsect_control_nonlineaire}. Dealing with a potential $V$ instead of $V=0$ is done with literally the same arguments.

\smallskip
\noindent
To highlight the novelties of this work, we postpone the proof of Proposition~\ref{prop_controle_traj_ref} to Section~\ref{subsect_traj_ref} and first prove how this proposition implies Theorem~\ref{th_controle_exact_local}. We start with the proof of  Theorem~\ref{th_controle_exact_local} in the particular case $C_0=C_f=I_N$, where~$I_N$ is the $N\times N$ identity matrix. This is done using Proposition~\ref{prop_controle_traj_ref}, a rotation phenomenon  for the solution corresponding to the null control on a suitable time interval,  and a time reversibility argument. Then, for any $C \in U_N$, using a linearity argument, we prove Theorem~\ref{th_controle_exact_local} in the case $C_0=C_f=C$. We end the proof using connectedness of the set of unitary matrices and a compactness argument.

\begin{proof}[Proof of Theorem~\ref{th_controle_exact_local}.]
To simplify notations, until the end of Section~\ref{sect_control_exact}, we shall write $\lambda_k$, $\varphi_k$ instead of $\lambda_{k,V}$, $\varphi_{k,V}$.

\medskip
\textit{First step : } proof in the case $C_0=C_f=I_N$.

\noindent
Let us take any $T>0$. Let  $\delta>0$ and $\theta_1,\dots,\theta_N$ be the constants  given in Proposition~\ref{prop_controle_traj_ref}. Let $\bo \psi_0, \bo \psi_f \in \OO_{\delta,I_N}$.

\noindent
As $\mathcal{O}_{\delta,I_N}= \mathcal{O}_\delta^0$, there exists  $u \in L^2((0,T),\R)$ such that the associated solution of (\ref{syst_Neq}) with initial condition~$\bo \psi_0$ satisfies
\begin{equation} \label{time_reversibility_1ere_etape}
\bo \psi(T) = (e^{i\theta_1} \varphi_{1} , \dots, e^{i\theta_N} \varphi_{N}).
\end{equation}
Using Condition $\mathrm{(C_5)}$ and the Kronecker theorem on diophantine approximation (see e.g.~\cite[Corollary 10]{Schmidt}), there exists a rotation time $T_r>0$ such that
\begin{equation*}
|\lambda_{j}|^{3/2} \left| e^{i(2 \theta_j - \lambda_{j} T_r)} - 1 \right| < \frac{\delta}{N}, \quad \forall j \in \{1,\dots,N\}.
\end{equation*}
Thus, it comes that
$\begin{displaystyle}
\sum_{j=1}^N \| e^{i(\theta_j - \lambda_{j} T_r)} \varphi_{j} - e^{-i \theta_j} \varphi_{j}\|_{H^3_{(V)}} < \delta.
\end{displaystyle}$
Together with (\ref{time_reversibility_1ere_etape}), this implies that if we extend $u$ by zero on $(T,T+T_r)$ then
\begin{equation*}
\sum_{j=1}^N \| \psi^j(T+T_r) - e^{-i\theta_j} \varphi_{j} \|_{H^3_{(V)}} < \delta.
\end{equation*}
Thus,
\begin{equation} \label{time_reversibility_2e_etape}
\overline{\bo \psi}(T+T_r) \in \mathcal{O}_\delta^f.
\end{equation}
As $\bo \psi_f \in \mathcal{O}_{\delta,I_N}= \mathcal{O}_\delta^0$ and the eigenvectors $\varphi_j$ being real-valued, we have $\overline{\bo \psi}_f \in \mathcal{O}_\delta^0$. Then, Proposition~\ref{prop_controle_traj_ref} implies the existence of $v \in L^2((0,T),\R)$ such that the associated solution of (\ref{syst_Neq}) with initial condition $ \overline{\bo \psi}_f$ equals to  $\overline{\bo \psi}(T+T_r)$ at time $T$.
Finally, the time reversibility property proves that if $u$ is defined  by $u(T+T_r+t) = v(T- t)$ for $t\in (0,T)$, then the associated solution of (\ref{syst_Neq}) with initial condition $\bo \psi_0$ satisfies
\begin{equation} \label{time_reversibility_3e_etape}
\bo \psi (T+T_r+T) = \bo \psi_f.
\end{equation}
This ends the proof of Theorem~\ref{th_controle_exact_local} in the case $C_0=C_f=I_N$ in time $T^* :=2T+T_r$.

\medskip
\textit{Second step :} proof in the case $C_0=C_f=C$.

\noindent
Let $\delta>0$ be as in the first step, $C \in U_N$, and $\bo z := C \bo \varphi$. Let $\delta_{\bo z}>0$ be sufficiently small to satisfy
\begin{equation*}
C^* \Big(B_{\bo H^3_{(V)}}(\bo z,\delta_{\bo z}) \Big) \subset B_{\bo H^3_{(V)}}(\bo \varphi, \delta).
\end{equation*}
Let us take any $\bo \psi_0, \bo \psi_f \in \mathcal{O}_{\delta_{\bo z},C}$ and  define 
\begin{equation} \label{cible_mat_unitaire}
\tilde{\bo \psi}_0 := C^* \bo \psi_0, \quad
\tilde{\bo \psi}_f := C^* \bo \psi_f.
\end{equation}
The unitarity of $C$ implies     that $\lag \tilde{\psi}^j_0, \tilde{\psi}^k_0\rag = \delta_{j=k}$ and  $\lag \tilde{\psi}^j_f, \tilde{\psi}^k_f\rag = \delta_{j=k}$. 
Thus, from the definition of $\delta_{\bo z}$ it follows that  $\tilde{\bo \psi}_0, \tilde{\bo \psi}_f \in \mathcal{O}_{\delta, I_N}$. Then, by the first step, there is a control $u \in L^2((0,T^*),\R)$ such that 
$$
\bp (T^*,\tilde{\bo \psi}_0,u)=\tilde{\bo \psi}_f.
$$   
Since system (\ref{syst_Neq}) is linear with respect to the state, the resolving operator commutes with $C$. Thus,  in view of (\ref{cible_mat_unitaire}), we have  
\begin{equation}\label{E:com}
\bp (T^*, {\bo \psi}_0,u)=\bp (T^*, C\tilde{\bo \psi}_0,u)=C\bp (T^*,  \tilde{\bo \psi}_0,u)=C\tilde{\bo \psi}_f={\bo \psi}_f.
\end{equation}This ends the proof the second step.

\medskip
\textit{Third step :} conclusion.

\noindent
Since $U_N$ is connected, there is a continuous mapping $t \in [0,1] \mapsto C(t) \in U_N$ with $C(0) = C_0$ and $C(1) = C_f$. By the previous step, for any $\bo z \in F := \left\{ C(t) \bo \varphi \, ; \, t \in [0,1] \right\}$, there is   $\delta_{\bo z}>0$ such that   (\ref{syst_Neq})   is exactly controllable  in~$B_{\bo H^3_{(V)}}(\bo z,\delta_{\bo z}) $  in time $T^*$. Using the compactness of the set $F$, we get the existence of $\bo z_j \in F$, $j=1, \dots,L$ with $L \in \N^*$ such that 
\begin{equation*}
F \subset \bigcup_{j=0}^L B_{\bo H^3_{(V)}}(\bo z_j , \delta_{\bo z_j}).
\end{equation*}
Without loss of generality, we can assume that $\bo z_L = \bo z_f$.
Finally, setting $T := (L+1) T^*$ and $\de:=\min\{\delta_{\bo z_0}, \delta_{\bo z_f}\}$, we see that for any  $\bo \psi_0 \in \OO_{\delta,C_0}$ and  $\bo \psi_f \in \mathcal{O}_{\delta,C_f}$ there is   a control $u \in L^2((0,T),\R)$ such that 
$$
\bp (T, {\bo \psi}_0,u)= {\bo \psi}_f
$$
This completes the proof of Theorem \ref{th_controle_exact_local}.
\end{proof}

\noindent
The rest of this section is dedicated to the proof of Proposition~\ref{prop_controle_traj_ref}.

\subsection{Construction of the reference trajectory}
\label{subsect_traj_ref}

\noindent
The proof of Proposition~\ref{prop_controle_traj_ref} relies on the return method introduced by Coron (see \cite[Chapter 6]{CoronBook} for a comprehensive introduction). 
The natural strategy to obtain local exact controllability around $\bo \varphi$ is to prove controllability for the linearized system
\begin{equation} \label{linearise_Neqs_eigenstates}
\left\{
\begin{aligned}
&i \partial_t	\Psi^j = \left(-\partial^2_{xx} + V(x) \right) \Psi^j - v(t) \mu(x) \Phi_j,  &(t,x)& \in (0,T) \times (0,1),
\\
&\Psi^j(t,0)= \Psi^j(t,1) = 0, &j& \in \{1,\dots ,N\},
\\
&\Psi^j(0,x) = 0.
\end{aligned}
\right.
\end{equation}
However, straightforward computations lead to 
\begin{equation} \label{linearise_Neqs_NC}
\lag \mu \varphi_k, \varphi_k \rag \lag \Psi^j(T), \Phi_j(T) \rag = 
\lag \mu \varphi_j, \varphi_j \rag \lag \Psi^k(T), \Phi_k(T) \rag,
\quad \forall j,k \in \{1, \dots, N\}.
\end{equation}
Thus, the linearized system (\ref{linearise_Neqs_eigenstates}) is not controllable and we use the return method.
In our setting, the main idea of this method is to design of a reference control $u_{ref}$ such that the associated solution $\bo \psi_{ref}$ of system (\ref{syst_Neq}) with initial condition $\bo \varphi$ satisfies 
\begin{equation*}
\bo \psi_{ref}(T) = (e^{i \theta_1} \varphi_{1}, \dots, e^{i \theta_N} \varphi_{N}),
\end{equation*} 
for some $\theta_1, \dots, \theta_N \in \R$ and  the linearized system around this trajectory is controllable. Then, an application of the inverse mapping theorem leads to local controllability of (\ref{syst_Neq}) around the trajectory $(u_{ref}, \bo \psi_{ref})$ and proves  Proposition~\ref{prop_controle_traj_ref}. The main ideas of this proof are adapted from \cite[Theorem 1.5]{Morancey_simultane}.  For the sake of completeness, we precise the adaptations that have been made and give a sketch of the proofs.
The reference trajectory is designed in the following  proposition.
\begin{proposition}
\label{prop_traj_ref}
Assume that Conditions $\mathrm{(C_3)}$ and $\mathrm{(C_4)}$ are satisfied for $V, \mu \in H^3((0,1),\R)$. Let $T>0$ and $0 < \varepsilon_0 < \dots < \varepsilon_{N-1} =: \varepsilon < T$. There exist $\overline{\eta} >0$ and $C>0$ such that for every $\eta \in (0, \overline{\eta})$, there are $\theta_1^\eta, \dots, \theta_N^\eta \in \R$ and a control $u_{ref}^{\eta} \in L^2((0,T),\R)$ with
\begin{equation} \label{traj_ref_borne_controle}
\|u_{ref}^{\eta}\|_{L^2(0,T)} \leq C \eta
\end{equation}
such that the associated solution $\bo \psi_{ref}^{\eta}$ of (\ref{syst_Neq}) with initial condition $\bo \varphi$ satisfies for $j \in \{1,\dots,N\}$ and  $k \in \{1,\dots,N-1\}$
\begin{equation} \label{traj_ref_condition_minimalite}
\lag \mu \psi^{j,\eta}_{ref}(\varepsilon_k),\psi^{j,\eta}_{ref}(\varepsilon_k) \rag = \lag \mu \varphi_{j}, \varphi_{j} \rag + \eta \delta_{j=k} , 
\end{equation}
and
\begin{equation} \label{traj_ref_etat_final}
\bo \psi_{ref}^{\eta}(T) = \big( e^{i \theta_1^\eta} \varphi_{1}, \dots, e^{i \theta_N^\eta} \varphi_{N} \big).
\end{equation}
\end{proposition}

\begin{remark}
As in~\cite{Morancey_simultane}, the conditions (\ref{traj_ref_condition_minimalite}), together with an appropriate choice  of the parameter $\eta$,   will imply the controllability of the linearized system around this reference trajectory (see Section~\ref{subsect_control_linearise}).
\end{remark}
\medskip

\begin{proof}[Sketch of the proof of Proposition~\ref{prop_traj_ref}.] We split the proof in two steps. In   the first step, we  construct   $u_{ref}^{\eta}$ on $(0,\varepsilon)$ such that (\ref{traj_ref_condition_minimalite}) is satisfied. Then in the second step, we extend   $u_{ref}^{\eta}$ to $(\varepsilon,T)$ in a such way that (\ref{traj_ref_etat_final}) is verified.

\medskip
\textit{First step :}
Let us take $u_{ref}^{\eta} \equiv 0$ on $[0, \varepsilon_0)$. Following the proof of~\cite[Proposition 3.1]{Morancey_simultane},  we construct a control $u_{ref}^{\eta}$ such that condition~(\ref{traj_ref_condition_minimalite}) is satisfied and  
\begin{equation} \label{traj_ref_borne_controle1}
\| u_{ref}^{\eta} \|_{L^2(\varepsilon_0,\varepsilon)} \leq C \eta,
\end{equation}
   by an application of the inverse mapping theorem  to the map
\begin{equation*}
\begin{array}{cccc}
\tilde{\Theta}: & L^2((\varepsilon_0,\varepsilon),\R) 
& \rightarrow & \R^N \times \dots \times \R^N
\\
& u & \mapsto & \left(  \tilde{\Theta}_1(u) , \dots, \tilde{\Theta}_{N-1}(u) \right)
\end{array}
\end{equation*}
at the point $u=0$, where
\begin{equation*}
\tilde{\Theta}_k(u) := 
\left(
\lag \mu \psi^j (\varepsilon_k) , \psi^j (\varepsilon_k) \rag - \lag \mu \varphi_j , \varphi_j \rag
\right)_{1 \leq j \leq N}.
\end{equation*}

\noindent
The $C^1$ regularity of $\tilde{\Theta}$ follows from the differentiability property in Proposition~\ref{P:2.1}. A continuous right-inverse of $\md \tilde{\Theta}(0)$ is constructed  by a resolution of a suitable trigonometric moment problem using Proposition~\ref{prop_pb_moment}.

\medskip
\textit{Second step :} For any $j \in \N^*$, let $\PP_{j}$ be the orthogonal projection defined by (\ref{def_proj}). We prove that for any initial condition at time $\varepsilon$ close enough to $\big( \Phi_{1} ,\dots, \Phi_{N} \big)(\varepsilon) $, the projections $\big( \PP_{1}(\psi^1(T)),\dots, \PP_{N}(\psi^N(T)) \big)$ can be brought to $0$ by a small   control $u \in L^2((\varepsilon,T),\R)$. This is sufficient to prove Proposition~\ref{prop_traj_ref}. Indeed, if
\begin{equation}
\label{traj_ref_proj_nulles}
\PP_{1} \big( \psi^{1,\eta}_{ref}(T) \big) = \dots = \PP_{N} \big( \psi^{N,\eta}_{ref}(T) \big) = 0,
\end{equation}
using the invariants (\ref{invariants}), it comes that there exist $\theta_1^\eta, \dots, \theta_N^\eta \in \R$ such that (\ref{traj_ref_etat_final}) holds.

\noindent
As in \cite[Proposition 3.2]{Morancey_simultane}, the condition (\ref{traj_ref_proj_nulles}) with a control satisfying 
\begin{equation} \label{traj_ref_borne_controle2}
\|u_{ref}^{\eta}\|_{L^2(\varepsilon,T)} \leq C \eta
\end{equation} 
is obtained by an application of the inverse mapping theorem to the map
\begin{equation*}
\Theta:    L^2((\varepsilon,T),\mathbb{R}) \times \bo H^3_{(0)}
  \rightarrow   \bo H^3_{(0)} \times \bo X,
\end{equation*}
at the point $\big(0, \Phi_{1}(\varepsilon), \dots, \Phi_{N}(\varepsilon) \big)$,  where
\begin{equation*}
\Theta \big( u, \bo \psi_0 \big) := 
\Big( \bo \psi_0, \, \PP_{1} \big( \psi^1(T) \big), \dots, \PP_{N} \big( \psi^N(T) \big) \Big)
\end{equation*}
and
\begin{equation}
\label{def_X1}
\bo X := \left\{ \bo \phi \in {\bo H}^3_{(0)} \, ; \, \lag \phi^j , \varphi_{k} \rag =0  \, \text{ for all } 1 \leq k \leq j \leq N \right\}.
\end{equation}
Again, the $C^1$ regularity of $\Theta$ is obtained thanks to    Proposition~\ref{P:2.1}. The continuous right-inverse of $\md \Theta \big(0, \Phi_{1}(\varepsilon), \dots, \Phi_{N}(\varepsilon) \big)$ is given by the resolution of a suitable trigonometric moment problem with frequencies
\begin{equation*}
\left\{ \lambda_k - \lambda_j \, ; \, j \in \{1, \dots, N\}, \, k \geq j+1 \right\}.
\end{equation*}
The solution of that moment problem is given by Proposition~\ref{prop_pb_moment}.

\end{proof}

\subsection{Controllability of the linearized system}
\label{subsect_control_linearise}

This section is dedicated to the proof of controllability of the following system which is the linearization of \eqref{syst_Neq}   around the reference trajectory  $\bo \psi^{\eta}_{ref}$: 
\begin{equation} \label{linearise_1D_traj_ref}
\left\{
\begin{aligned}
& i \partial_t \Psi^j = \left(-\partial^2_{xx} +V(x) \right) \Psi^j - u_{ref}^{\eta}(t) \mu(x) \Psi^j - v(t) \mu(x) \psi^{j,\eta}_{ref}, 
\\
& \Psi^j(t,0) = \Psi^j(t,1) = 0,  \quad\quad \quad\quad\quad\quad j \in \{1, \dots,N\},
\\
& \Psi^j(0,x) = \Psi^j_0(x). 
\end{aligned}
\right.
\end{equation}
For any $t\in [0,T]$, let us define the following space
\begin{align*}
\bo X_t :&= \Big\{ \bo \phi \in {\bo H}^3_{(0)} \, ; \, \Re (\lag \phi^j, \psi^{j,\eta}_{ref}(t) \rag) = 0  \text{ for } j=1,\dots,N 
\\ & \text{and } \lag \phi^j, \psi^{k,\eta}_{ref}(t) \rag = - \overline{\lag \phi^k ,\psi^{j,\eta}_{ref}(t)\rag} \text{ for } j=2,\dots,N, \, k<j  \Big\}.
\end{align*}
This space is given by the linearization of the invariants (\ref{invariants}) around the reference trajectory.

We prove the following controllability result.

\begin{proposition} \label{prop_controle_linearise_1D}
There exists $\hat{\eta} \in (0,\overline{\eta})$ such that for any $\eta \in (0,\hat{\eta})$, there exists a continuous linear map
\begin{equation*}
\begin{array}{cccc}
L^\eta: & \bo X_0 \times \bo X_T
& \rightarrow & L^2((0,T),\mathbb{R})
\\
& \big( \bo \Psi_0, \bo \Psi_f \big) & \mapsto & v
\end{array}
\end{equation*}
such that for any $\bo \Psi_0 \in \bo X_0$ and  $\bo \Psi_f \in \bo X_T$, the solution $\bo \Psi$ of system (\ref{linearise_1D_traj_ref}) with initial condition $\bo \Psi_0$ and control $v := L^\eta(\bo \Psi_0, \bo \Psi_f)$ satisfies $\bo \Psi(T) = \bo \Psi_f$.
\end{proposition}
The proof of Proposition~\ref{prop_controle_linearise_1D} is adapted from~\cite[Proposition 4.1]{Morancey_simultane}. As the proof is quite long and technical, we recall the main steps and arguments.
Let us set some notations that will be used throughout this proof. For any~$\eta \in (0,\overline{\eta})$ and $k\in \N^*$, let $\Phi^\eta_k = \psi(\cdot,\varphi_k,u_{ref}^\eta)$ as defined by (\ref{def_propagateur}). Notice that for~$j \in \{1,\dots,N\}$, $\Phi^\eta_j = \psi^{j,\eta}_{ref}$ and for any $t \in [0,T]$, $\{\Phi^\eta_k(t)\}_{k\in \N^*}$ is a Hilbert basis of $L^2((0,1),\C)$, as an image of a Hilbert basis by a unitary operator. Let
\begin{equation*}
\II := \big\{ (j,k) \in \{1,\dots,N\} \times \N^* \, ; \, k \geq j+1 \big\} \cup \left\{ (N,N) \right\}.
\end{equation*}

In the  first step we prove  the controllability of the directions $\lag \Psi^j(T), \Phi^\eta_k(T)\rag$ for $(j,k) \in \II$ for $\eta$ small enough. This comes from the solvability of the trigonometric moment problem associated to the case $\eta=0$ and a close linear maps argument.
Then, we exhibit a minimal family that allows to control, simultaneously to the previous direction, the remaining diagonal directions $\lag \Psi^j(T) ,\Phi^\eta_j(T) \rag$ for $j \in \{1,\dots,N-1\}$. This is the main feature of the design of the reference trajectory. Indeed, we enlightened in (\ref{linearise_Neqs_NC}) that those diagonal directions were the ones leading to non controllability of the linearized system in the case $\eta=0$.
Finally, due to the definition of $\bo X_T$, the remaining directions $\lag \Psi^j(T), \Phi^\eta_k(T)\rag$ for $1 \leq k < j$ are automatically controlled.
\medskip

\begin{proof}[Sketch of the  proof of Proposition~\ref{prop_controle_linearise_1D}.]
Let $R : \II \to \N$ be the rearrangement such that, if $\omega_{n} := \lambda_{k}- \lambda_{j}$ with $n=R(j,k)$, the sequence $(\omega_{n})_{n \in \N}$ is increasing. Notice that $0=R(N,N)$.

\textit{First step.} Let us take any  $T_f \in (0,T]$ and prove that there is $\hat{\eta}=\hat{\eta}(T_f) \in (0,\overline{\eta})$ such that for any $\eta \in (0,\hat{\eta})$ there exists a continuous linear map
\begin{equation*}
G^\eta_{T_f}:  \bo X_0 \times \ell^2_r(\N,\C) \rightarrow  L^2((0,T_f),\mathbb{R})
\end{equation*}
such that for any $\bo \Psi_0 \in \bo X_0$, $d=(d_n)_{n\in \Z} \in \ell^2_r(\N,\C)$, the solution $\bo \Psi$ of system (\ref{linearise_1D_traj_ref}) with initial condition $\bo \Psi_0$ and control $v=G^\eta_{T_f}(\bo \Psi_0, d)$ satisfies
\begin{equation*}
\frac{\lag \Psi^j(T_f), \Phi^\eta_k(T_f) \rag}{i \lag \mu \varphi_{j} , \varphi_{k} \rag} = d_n, \quad \forall (j,k) \in \II, \, n=R(j,k).
\end{equation*}
Let
\begin{equation*}
f^\eta_n :t \in [0,T] \mapsto \frac{\lag \mu \psi^{j,\eta}_{ref}(t) , \Phi^\eta_k(t) \rag}{\lag \mu \varphi_{j}, \varphi_{k} \rag}  \quad \text{ for } (j,k) \in \II \text{ and } n=R(j,k),
\end{equation*}
$f^\eta_{-n} := \overline{f^\eta_n}$ for $n \in \N^*$ and $H_0 := \text{Adh}_{L^2(0,T_f)}\big( \text{Span} \{ e^{i \omega_{n} \cdot} , n \in \Z\} \big)$.
As in \cite[Lemma 4.1]{Morancey_simultane}, the construction of $G^\eta_{T_f}$ relies on the fact that the map
\begin{equation*}
\begin{array}{cccc}
J^\eta: &  L^2((0,T_f),\C)
& \rightarrow & \ell^2(\Z,\C)
\\
& v & \mapsto &  \left( \int_0^{T_f} v(t) f^\eta_n(t) \md t \right)_{n \in \Z}
\end{array} 
\end{equation*}
is an isomorphism from $H_0$ to $\ell^2(\Z,\C)$.
Indeed, for any $(j,k) \in \II$ and $n=R(j,k)$, straightforward computations lead to
\begin{equation*}
\lag \Psi^j(T_f), \Phi^\eta_k(T_f) \rag = \lag \Psi^j_0,\varphi_k \rag + i \lag \mu \varphi_{j}, \varphi_{k} \rag \int_0^{T_f} v(t) f^\eta_n(t) \md t.
\end{equation*}
The isomorphism property of $J^\eta$ comes from the estimate 
\begin{equation*}
\| J^\eta - J^0 \|_{\mathcal{L}(L^2(0,T_f),\ell^2)} \leq C \|u_{ref}^{\eta}\|_{L^2(0,T_f)} \leq C \eta,
\end{equation*}
(see \cite[Proof of Lemma 4.1]{Morancey_simultane} for the proof of this estimate) and the fact that, due to Proposition~\ref{prop_pb_moment}, $J^0$ is an isomorphism from $H_0$ to $\ell^2(\Z,\C)$.
\medskip

\textit{Second step.} Let $\hat{\eta} < \min(\hat{\eta}(T),\hat{\eta}(\varepsilon_0))$ with $\varepsilon_0$ as in Proposition~\ref{prop_traj_ref}. In all what follows we assume $\eta \in (0,\hat{\eta})$. Let 
\begin{equation} \label{def_fjj}
f^\eta_{j,j} : t\in [0,T] \mapsto \dfrac{\lag \mu \psi^{j,\eta}_{ref}(t), \psi^{j,\eta}_{ref}(t)\rag}{\lag \mu \varphi_{j}, \varphi_{j} \rag}  \quad \text{ for } j \in \{1, \dots, N-1 \}.
\end{equation} 
Then, the family $\Xi := (f^\eta_n)_{n \in \Z} \cup \{ f^\eta_{1,1} , \dots, f^\eta_{N-1,N-1} \}$ is minimal in $L^2((0,T),\C)$.
The proof of this is a straightforward extension of~\cite[Lemma 4.3]{Morancey_simultane} and is not detailed. It relies on the fact that  $(f^{\eta}_n)_{n \in \Z}$ is a Riesz basis of $\text{Adh}_{L^2(0,T)}\big( \text{Span} \{ f^\eta_n  , \linebreak[1] n \in \Z \} \big)$ and conditions (\ref{traj_ref_condition_minimalite}).
\medskip

\textit{Third step :} conclusion. From the second step, we get the existence of a biorthogonal family associated to $\Xi$ in $\text{Adh}_{L^2(0,T)}\big( \text{Span}\{\Xi\} \big)$ denoted by
\begin{equation} \label{def_famille_biorthogonale}
\big\{ g^\eta_{1,1}, \dots, g^\eta_{N-1,N-1}, (g^\eta_n)_{n \in \Z} \big\},  
\end{equation}
with $g^\eta_{j,j}$ being real-valued for $j \in \{1, \dots, N\}$.
The map $L^\eta$ is defined by 
\begin{equation*}
L^\eta : \big( \bo \Psi_0, \bo \Psi_f \big) \in \bo X_0 \times \bo X_T \mapsto v \in L^2((0,T),\R),
\end{equation*}
where 
\begin{equation*}
v := v_0 + \sum_{j=1}^{N-1} \Big( \frac{\Im(\lag \Psi^j_f , \psi^{j,\eta}_{ref}(T) \rag) - \Im(\lag \Psi^j_0, \varphi_j \rag)}{\lag \mu \varphi_j , \varphi_j \rag} - \int_0^T v_0(t) f^\eta_{j,j}(t) \md t \Big) g^\eta_{j,j}, 
\end{equation*}
and $v_0 := G^\eta_T(\bo \Psi_0, d(\bo \Psi_f))$ with 
$d(\bo \Psi_f)_n := \dfrac{\lag \Psi^j_f,\Phi^\eta_k(T)\rag}{i\lag \mu \varphi_{j}, \varphi_{k} \rag}$, for $(j,k) \in \II$ and $n=R(j,k)$. The biorthogonality properties and the first step imply
\begin{equation*}
\lag \Psi^j(T), \Phi^\eta_k(T) \rag = \lag \Psi^j_f, \Phi^\eta_k(T) \rag, \quad \forall (j,k) \in \II \cup \{ (1,1) , \dots, (N-1,N-1) \}.
\end{equation*}
Finally, for $j \in \{ 2, \dots, N \}$ and $k<j$ explicit computations lead to
\begin{equation*}
\lag \Psi^j(T), \psi^{k,\eta}_{ref}(T) \rag  = - \overline{\lag \Psi^k(T), \psi^{j,\eta}_{ref}(T) \rag}.
\end{equation*}
As $\bo \Psi_f \in \bo X_T$, this ends the proof of Proposition~\ref{prop_controle_linearise_1D}.

\end{proof}

\subsection{Controllability of the nonlinear system}
\label{subsect_control_nonlineaire}

In this subsection, we end the proof of Proposition~\ref{prop_controle_traj_ref}. We consider the reference trajectory designed in Proposition~\ref{prop_traj_ref}. Let $\hat{\eta}$ be given by Proposition~\ref{prop_controle_linearise_1D}. We assume in all what follows that $\eta \in (0,\hat{\eta})$ is fixed. Using the inverse mapping theorem and Proposition~\ref{prop_controle_linearise_1D}, we prove in Proposition~\ref{prop_controle_NL_1D} that the projections  onto the space $\bo X_T$ (see (\ref{def_projection_XfT}) for a precise definition) are exactly controlled. Then, using the invariants (\ref{invariants}) of the system, we prove that controlling these projections is sufficient to control the full trajectory. Let us set
\begin{equation} \label{def_Omega}
\bo \Omega := \big\{ \bo \phi \in \bo H^3_{(0)} \, ; \, \lag \phi^j , \phi^k \rag = \delta_{j=k}, \; \forall j,k \in \{1, \dots, N\} \big\}
\end{equation}
and  define
\begin{equation*}
\begin{array}{cccc}
\Lambda : &   \bo \Omega \times L^2((0,T),\R)
& \rightarrow & \bo \Omega \times \bo X_T 
\\
&(\bo \psi_0, u) & \mapsto & \big(\bo \psi_0 , \, \tilde{\PP}_1(\psi^1(T)), \dots, \tilde{\PP}_N(\psi^N(T)) \big),
\end{array}
\end{equation*}
where $\bo \psi := \bo\psi(\cdot, \bo \psi_0,u)$ and
\begin{align*}
\tilde{\PP}_j(\phi^j) :&= \phi^j - \Re \big( \lag \phi^j, \psi^{j,\eta}_{ref}(T) \rag \big) \, \psi^{j,\eta}_{ref}(T) \,
\\
- &\, \sum_{k=1}^{j-1} \big( \lag \phi^j , \psi^{k,\eta}_{ref}(T) \rag + \lag \psi^{j,\eta}_{ref}(T),\phi^k \rag \big) \, \psi^k_{ref}(T).
\label{def_projection_XfT}
\tag{\theequation} \addtocounter{equation}{1}
\end{align*}Thus, $\Lambda$ takes value in $\bo \Omega \times \bo X_T$ and $\Lambda(\bo\varphi,u_{ref}^{\eta})= (\bo\varphi,0)$. 
The following proposition holds.

\begin{proposition} \label{prop_controle_NL_1D}
There exist $\tilde{\delta}>0$ and a $C^1$ map 
\begin{equation*}
\Upsilon : \OO_{\tilde{\delta}}^0 \times \tilde{\OO}_{T,\tilde{\delta}} \to L^2((0,T),\R),
\end{equation*}
where $\OO_{\tilde{\delta}}^0$ is defined in Proposition~\ref{prop_controle_traj_ref} and
\begin{equation*}
\tilde{\OO}_{T,\tilde{\delta}} := \Big\{ \tilde{\bo \psi}_f \in \bo X_T \, ; \,
\sum_{j=1}^N \| \tilde{\psi}^j_f \|_{H^3_{(V)}} < \tilde{\delta} \Big\},
\end{equation*}
such that $\Upsilon \big(\bo \varphi, \bo 0 \big) = u_{ref}^{\eta}$ and for any $\bo \psi_0 \in \OO_{\tilde{\delta}}^0$ and  $\tilde{\bo \psi}_f  \in \tilde{\OO}_{T,\tilde{\delta}}$ the solution $\bo \psi$ of system (\ref{syst_Neq}) with initial condition $\bo \psi_0$ and control $u= \Upsilon \big(\bo \psi_0,\tilde{\bo \psi}_f \big)$ satisfies
\begin{equation*}
\big( \tilde{\PP}_1(\psi^1(T)), \dots, \tilde{\PP}_N(\psi^N(T)) \big) = \tilde{\bo \psi}_f.
\end{equation*}
\end{proposition}

\begin{proof}[Sketch of proof.]
As~\cite[Proposition 4.2]{Morancey_simultane}, this proposition is proved by an  application of the inverse mapping theorem to the map $\Lambda$ at the point $(\bo \varphi,u_{ref}^{\eta})$. This map is $C^1$ by Proposition~\ref{prop_pb_moment}, and a continuous right inverse of the map
\begin{equation*}
\md \Lambda\big(\bo\varphi, u_{ref}^{\eta} \big): \bo X_0 \times L^2((0,T),\R) \to \bo X_0 \times \bo X_T
\end{equation*}
is given by Proposition~\ref{prop_controle_linearise_1D}.

\end{proof}
Finally, we prove Proposition~\ref{prop_controle_traj_ref}.
\begin{proof}[Proof of Proposition~\ref{prop_controle_traj_ref}.]
Let  us take any $\bo\psi_0 \in \OO_{\delta}^0$ and $\bo \psi_f \in \OO_{\delta}^f$, where the sets $\OO_{\delta}^0$ and~$  \OO_{\delta}^f$ are defined    in Proposition~\ref{prop_controle_traj_ref} and $\delta >0$ will be specified later on.
Let $\tilde{\delta}$ be the constant in Proposition~\ref{prop_controle_NL_1D} and
\begin{equation*}
\tilde{\bo \psi}_f := \big( \tilde{\PP}_1(\psi^1_f), \dots, \tilde{\PP}_N(\psi^N_f) \big).
\end{equation*}
  For sufficiently small   $\delta \in (0,\tilde{\delta})$, we have $\tilde{\bo \psi}_f \in \tilde{\OO}_{T,\tilde{\delta}}$    and
\begin{equation} \label{Re_cible>0}
\Re (\lag \psi^j_f , \psi^{j,\eta}_{ref}(T) \rag) > 0, \quad \forall j \in \{1,\dots,N\},
\end{equation} 
for any $\bo \psi_f \in \OO_{\delta}^f$. 
Let $u:=\Upsilon\big(\bo \psi_0,\tilde{\bo \psi}_f \big)$ and let $\bo \psi$ be the associated solution of (\ref{syst_Neq}) with initial condition~$\bo \psi_0$. We prove that (up to an a priori reduction of~$\delta$)
\begin{equation} \label{final=cible}
\bo \psi (T) = \bo \psi_f.
\end{equation}
Thanks to the regularity of $\Upsilon$ and Proposition~\ref{P:2.1}, it comes that, up to a reduction of $\delta$, one can assume that 
\begin{equation} \label{Re_final>0}
\Re(\lag \psi^j(T), \psi^{j,\eta}_{ref}(T) \rag) >0, \quad \forall j \in \{1, \dots, N\}.
\end{equation}
By Proposition~\ref{prop_controle_NL_1D}, we get
\begin{equation*}
\psi^1(T) - \Re(\lag \psi^1(T),\psi^{1,\eta}_{ref}(T) \rag) \psi^{1,\eta}_{ref}(T) =
\psi^1_f - \Re(\lag \psi^1_f, \psi^{1,\eta}_{ref}(T) \rag) \psi^{1,\eta}_{ref}(T).
\end{equation*}
Thus, using the fact that $\|\psi^1(T)\| = \|\psi^1_f\| $ and (\ref{Re_cible>0}), (\ref{Re_final>0}), we get $\psi^1(T) = \psi^1_f$. Assume that 
\begin{equation*}
\big( \psi^1 ,\dots, \psi^{j-1} \big) (T) = \big( \psi^1_f, \dots, \psi^{j-1}_f \big)\quad\text{for } j \in \{2, \dots,N\}.
\end{equation*}
Then the equality $\tilde{\PP}_j(\psi^j(T))= \tilde{\psi}^j_f$ gives
\begin{align*}
&\psi^j(T) - \Re(\lag \psi^j(T), \psi^{j,\eta}_{ref}(T) \rag) \psi^{j,\eta}_{ref}(T) - \sum_{k=1}^{j-1} \lag \psi^j(T), \psi^{k,\eta}_{ref}(T) \rag \psi^{k,\eta}_{ref}(T)
\\
&= \psi^j_f - \Re(\lag \psi^j_f, \psi^{j,\eta}_{ref}(T) \rag) \psi^{j,\eta}_{ref}(T) - \sum_{k=1}^{j-1} \lag \psi^j_f, \psi^{k,\eta}_{ref}(T) \rag \psi^{k,\eta}_{ref}(T).
\label{egalites_projections_1D}
\tag{\theequation} \addtocounter{equation}{1}
\end{align*}
Taking the scalar product of (\ref{egalites_projections_1D}) with $\psi^n(T) (= \psi^n_f)$ for $n \in \{1, \dots, j-1\}$ and using the constraints $\lag \psi^j(T), \psi^k(T) \rag = \lag \psi^j_f, \psi^k_f \rag = \delta_{j=k}$, we get  
\begin{align*}
&\Re(\lag \psi^j(T), \psi^{j,\eta}_{ref}(T) \rag) \, \lag \psi^{j,\eta}_{ref}(T), \psi^n_f \rag 
+ \sum_{k=1}^{j-1} \lag \psi^j(T), \psi^{k,\eta}_{ref}(T) \rag \, \lag \psi^{k,\eta}_{ref}(T), \psi^n_f \rag
\\
&=\Re(\lag \psi^j_f, \psi^{j,\eta}_{ref}(T) \rag) \, \lag \psi^{j,\eta}_{ref}(T), \psi^n_f \rag 
+ \sum_{k=1}^{j-1} \lag \psi^j_f, \psi^{k,\eta}_{ref}(T) \rag \, \lag \psi^{k,\eta}_{ref}(T), \psi^n_f \rag.
\end{align*}
Straightforward algebraic manipulations of these equations lead to the existence of $\gamma_1, \dots, \gamma_{j-1} \in \C$ that are proved to be arbitrarily small (up to an a priori reduction of $\delta$) such that for $k \in \{1, \dots, j-1\}$
\begin{align*} 
\lag \psi^j(T), \psi^{k,\eta}_{ref}(T) \rag &= \lag \psi^j_f, \psi^{k,\eta}_{ref}(T) \rag 
\\
&+ \gamma_k \Big( \Re(\lag \psi^j(T), \psi^{j,\eta}_{ref}(T) \rag) - \Re(\lag \psi^j_f, \psi^{j,\eta}_{ref}(T) \rag) \Big).
\label{controle_NL_1D}
\tag{\theequation} \addtocounter{equation}{1}
\end{align*}
If the $\gamma_j$'s are small enough this is consistent with $\| \psi^j(T) \|  = \| \psi^j_f \| $ only if 
\begin{equation*}
\Re(\lag \psi^j(T), \psi^{j,\eta}_{ref}(T) \rag) = \Re(\lag \psi^j_f, \psi^{j,\eta}_{ref}(T) \rag).
\end{equation*}
Together with (\ref{controle_NL_1D}), this implies $\psi^j(T) = \psi^j_f$ and ends the proof of Proposition~\ref{prop_controle_traj_ref}.

\end{proof}

\section{Global exact controllability}
\label{sect_control_global}

\subsection{Global exact controllability under favourable hypothesis}

In this section, combining the properties of approximate controllability  proved in Theorem~\ref{Approx_contr} and   local exact controllability proved in Theorem~\ref{th_controle_exact_local}, we establish global exact controllability for \eqref{syst_Neq}, under the following hypotheses on the functions $V, \mu \in H^4((0,1),\R)$
 \begin{description}
\item[$  \boldsymbol{\mathrm{(C_6)}}$]  For any $j\in \N^*$, there exists $C_j > 0$ such that
 $$|\lag \mu \varphi_{j,V}, \varphi_{k,V} \rag| \ge \frac{C_j }{k^{3}}\quad\text{for all $k \in \N^*$.}$$  
 \item[$  \boldsymbol{\mathrm{(C_7)}}$]  The numbers $\{1,\la_{j,V}\}_{j\in \N^*}$ are rationally independent, i.e., for any $M\in \N^*$ and $\boldsymbol r\in \Q^{M+1}\backslash\{\boldsymbol 0\}$,  we have 
 $$
r_0 + \sum_{j=1}^M r_j \la_{j,V}\neq 0.
 $$
\end{description}
Notice that these conditions imply Conditions   $\mathrm{(C_1)}- \mathrm{(C_5)}$. 
 
 \begin{theorem}\label{T:4.1}
 Assume that Conditions~$\mathrm{(C_6)}$ and~$\mathrm{(C_7)}$ are satisfied for the functions $V,\mu\in H^4((0,1),\R)$.  Then, for any unitarily equivalent vectors $\bp_0,\bp_f\in\boldsymbol \SS\cap \boldsymbol H^{4}_{(V)} $,    there is a time $T>0$ and a control $u\in L^2((0,T),\R)$ such that the solution of  (\ref{syst_Neq}) satisfies 
 $$
 \bp(T)= \bp_f.
 $$ 
 \end{theorem}
\begin{proof} 
In this proof, we   use vectors of different size. In bold characters we   denote only the vectors of size $N$.  

\smallskip
\textit{First step}. Let us take any $M\in \N^*$ and $\bo z \in \boldsymbol \CC_M $ and  prove that there is a time $T>0$ and a constant $\delta>0$ such that for any $\bp_0,  \bp_f \in B_{\boldsymbol H^3_{(V)} }(\bo z,\delta)$ which are unitarily equivalent to $\bo z$, there is a control $u\in L^2((0,T),\R)$ satisfying  
$\bo \psi(T,\bp_0, u ) = \bp_f$. Here we use the following technical lemma whose proof is postponed to the end of this subsection.
\begin{lemma} \label{lemme_propagateur_H3}
For any $\bo z \in \bo \CC_M$ and  $\epsilon >0$, there is $\delta >0$ such that for any $\bo \phi \in B_{ \bo H^3_{(V)}}(\bo z, \de)$, which  is  unitarily equivalent to $\bo z$, there exists $\UU_\phi \in U(L^2)$ satisfying $\UU_\phi \bo z = \bo \phi$ and $\|\UU_\phi \varphi_j - \varphi_j \|_{H^3_{(V)}} < \epsilon$ for $j=1,\dots,M$.
\end{lemma}

Notice that under   Conditions~$\mathrm{(C_6)}$ and~$\mathrm{(C_7)}$, we can apply Theorem~\ref{th_controle_exact_local}  in the case of  $M$ equations and   $C_0=C_f=I_M$. We denote by $\de_{*}$ and $T_{*}$ the corresponding radius and time given in Theorem~\ref{th_controle_exact_local}.
Let $\delta$ be the constant in Lemma~\ref{lemme_propagateur_H3} corresponding to    $\epsilon = \frac{\delta_{*}}{M}$. 
Then for any $\bp_f \in B_{\bo H^3_{(V)} }(\bo z,\delta)$, which  is  unitarily equivalent to~$\bo z$, we have  $\sum_{j=1}^M\|\UU_{\psi_f} \varphi_j - \varphi_j \|_{H^3_{(V)}} < \delta_{*}$. Thus Theorem~\ref{th_controle_exact_local}  implies    the existence of a control $u_f \in L^2((0,T_{*}),\R)$ driving the solution of~(\ref{syst_Neq}) of size $M$ from $(\varphi_1,\dots,\varphi_M)$ to $\UU_{\psi_f} (\varphi_1, \dots,   \varphi_M)$. As $\bo z \in \bo \CC_M $, there exists a matrix $C\in \C^{N\times M}$ such that $\bo z = C (\varphi_1,\dots,\varphi_M)$. Then we have $$ C \UU_{\psi_f} (\varphi_1, \dots,   \varphi_M)=\UU_{\psi_f} C (\varphi_1, \dots,   \varphi_M)=\UU_{\psi_f}  \bo z=  \bo \psi_f.$$
Combining this with  the fact that (\ref{syst_Neq}) is linear with respect to the state, we get  that
 the control $u_f$ also drives the solution of (\ref{syst_Neq}) of size $N$ from $\bo z $ to $\bp_f $  (cf.~\eqref{E:com}).

\noindent
The same strategy leads to the existence of a control $u_0 \in L^2((0,T_{*}),\R)$ driving the solution of (\ref{syst_Neq}) of size $N$ from $\overline{\bo z}$ to $ \overline{\bp_0}$. Thus, using the time reversibility property and setting $T=2 T_{*}$, $u(t) = u_0(T_{*}-t)$ on $(0,T_{*})$ and $u(t)=u_f(t-T_{*})$ on $(T_{*},T)$, we end the proof of the  first step.

\medskip
\textit{Second step}.
Let $M \in \N^*$ and $\bo z_0, \bo z_f \in \bo \CC_M$ be unitarily equivalent. In this step, we prove that there  is a constant $\de>0$ and a time $T>0$ such that for any  $\bo\psi_0\in B_{\bo H^3_{(V)}}(\bo z_0, \de)$ and   $\bo\psi_f\in B_{\bo H^3_{(V)}}(\bo z_f, \de)$,  which are unitarily equivalent to~$\bo z_0$, there is
    a control $u \in L^2((0,T),\R)$ such that   $\bo \psi(T, \bo\psi_0, u) = \bo\psi_f$.

\noindent
As $\bo z_0, \bo z_f \in \bo\CC_M$, there exists $\UU \in U(\CC_M)$ such that $\bo z_f = \UU \bo z_0$. Since~$  U(\CC_M)$ is connected, we can choose a continuous
 mapping $t \in [0,1] \mapsto \UU(t) \in U(\CC_M)$ such that  $\UU(0) = I_M$ and $\UU(1) = \UU$. Then using the exact controllability result proved in the first step  for the vectors $\UU(t) \bo z_0,$~$t\in [0,1]$ and an    argument of   compactness, as in the third step of the proof of Theorem~\ref{th_controle_exact_local}, we get the required property.

\medskip
\textit{Third step}. Let us take any unitarily equivalent  $\bp_0,\bp_f\in\boldsymbol \SS\cap \boldsymbol H^{4}_{(V)}\cap \Be $   and prove that there is a time $T>0$ and a control $u\in L^2((0,T),\R)$ such that   $ \bp(T,\bp_0,u)= \bp_f $.
Applying Theorem~\ref{Approx_contr} to $\bp_0$ and  $\overline \bp_f$, we find sequences $T_{0n}, T_{fn}$ and $u_{0n}\in L^2((0,T_{0n}),\R), u_{fn}\in L^2((0,T_{fn}),\R)$ such that 
$$
\|\bp(T_{0n},  {\bp_0},u_{0n})-   \bp_{01}\|_{H^3_{(V)}} + \|\bp(T_{fn},  {\overline \bp_f},u_{fn})-    {\overline \bp_{f1}}\|_{H^3_{(V)}}  \underset{n \to \infty}{\longrightarrow} 0
$$ 
for some $\bp_{01},\bp_{f1}\in \boldsymbol\CC_M$. By the second step, we have exact controllability between some $\delta$-neighbourhoods of $\bp_{01}$ and $\bp_{f1}$ (notice that these vectors are unitarily equivalent). Choosing $n$ so large that  
$$
\|\bp(T_{0n},  {\bp_0},u_{0n})-   \bp_{01}\|_{H^3_{(V)}} + \|\overline{\bp(T_{fn},  {\overline \bp_f},u_{fn}})-    { \bp_{f1}}\|_{H^3_{(V)}}  < \de,
$$ 
we find a time $\tilde T $ and a control $\tilde u \in L^2((0,\tilde T),\R)$ such that 
$$
\bp(\tilde T ,  {\bp(T_{0n},  {\bp_0},u_{0n})},\tilde u )=\overline{\bp(T_{fn},  {\overline \bp_f},u_{fn}}).
$$
Taking $T=T_{0n}+\tilde T+T_{fn}$ and $u(t)=u_{0n}(t)$ for $t\in (0,T_{0n})$, $u(t)=\tilde u (t-  T_{0n})$ for $t\in (T_{0n},T_{0n}+ \tilde T)$, and  $u(t)=u_{fn}(T-t)$ for $t\in (T_{0n}+ \tilde T, T)$, and using the time reversibility property, we get $ \bp(T,\bp_0,u)= \bp_f$.

\medskip
\textit{Fourth step}. 
By the time reversibility property, to complete the proof of the theorem, it remains to show that for any $\bp_0\in\boldsymbol \SS\cap \boldsymbol H^{4}_{(V)} $ we have $\bp(T, \bp_0,u) \in \boldsymbol H^{4}_{(V)}\cap\Be $  for some $T>0$ and $u\in L^2((0,T),\R)$. 
 
Let us take any $\bp_{0n}, \bp_f\in\boldsymbol \SS\cap \boldsymbol H^{4}_{(V)} \cap \Be $ such that $\bp_{0n} \underset{n\to\ty}{\longrightarrow} \bp_{0}$ in $\boldsymbol L^2$. From the previous step, there are sequences $T_n$ and $u_n\in L^2((0,T_n),\R)$ such that $\bp(T_n, \bp_{0n},u_n)= \bp_f  $. Then
$$
\|\bp(T_n, \psi_{0}^j,u_n)-\bp_f\| = \|\bp_0-\bp_{0n}\|  \underset{n \to \infty}{\longrightarrow} 0,
$$
therefore
$$
\prod_{j=1}^N|\lag \psi(T_n, \psi_{0}^j,u_n),\varphi_{j}\rag |^2 \underset{n \to \infty}{\longrightarrow}   \prod_{j=1}^N|\lag \psi_f^j,\varphi_{j}\rag |^2 \neq 0.
$$ 
Thus $\bp(T_n, \bo\psi_{0},u_n)\in \Be$ for sufficiently large $n$. Finally, taking a control $\tilde{u} \in C^\infty_0((0,T_n),\R)$ sufficiently close to  $u$ in   $L^2((0,T_n),\R)$, we get $\bp(T, \bp_0,\tilde u) \in \boldsymbol H^{4}_{(V)}\cap\Be $.
This completes the proof of Theorem~\ref{T:4.1}.

\end{proof}

We end this section by the proof of Lemma~\ref{lemme_propagateur_H3}.
\begin{proof}[Proof of Lemma~\ref{lemme_propagateur_H3}.]
Let $\AA_{\bo \phi} := \text{Span} \{ \phi_i \, ; \, i=1, \dots,N \}$. As $\bo \phi$ and $\bo z$ are unitarily equivalent, there exists a linear map $L_{\bo \phi} : \AA_{\bo z} \to \AA_{\bo \phi}$ such that $L_{\bo \phi} \bo z= \bo \phi$ and
\begin{equation*}
\lag L_{\bo \phi} \xi, L_{\bo \phi} \zeta \rag = \lag \xi , \zeta \rag, \quad \forall \xi, \zeta \in \AA_{\bo z}. 
\end{equation*}
Let $\{\psi_k^z\}_{1\le k\le M}$ be an orthonormal basis in  $\CC_M$ (with respect to the scalar product in~$L^2$ ) such that    $\{\psi_k^z\}_{1\le k\le n}$ is  a basis in  $\AA_{\bo z}$.  If we define  $\psi_j^\phi := L_{\bo \phi} \psi_j^z$ for $j=1,\dots,n$, then $\{\psi_k^\phi\}_{1\le k\le n}$ will be an orthonormal basis in $\AA_{\bo \phi}$ and $\psi_j^\phi \underset{\bo \phi \to \bo z}{\longrightarrow} \psi_j^z$ in $H^3_{(V)}$ for $j=1, \ldots, n$.
 Let 
\begin{align*}
\tilde{\psi}_k^\phi &:= \psi_k^\phi, \quad \forall k \in \{1, \dots, n\},
\\
\tilde{\psi}_k^\phi &:= \psi_k^z - \sum_{j=1}^n \lag \psi_k^z,\psi_j^\phi \rag \psi_j^\phi, \quad \forall k \in \{ n+1, \dots,M\}.
\end{align*}
It is  easy to see that  $\tilde{\psi}_k^\phi \underset{\bo \phi \to \bo z}{\longrightarrow} \psi_k^z$ in $H^3_{(V)}$  for $k=1,\dots,M$. Thus if $\bo \phi$ is sufficiently close to $\bo z$ in $\bo H^3_{(V)}$, then $\{\tilde{\psi}_k^\phi\}_{1\le k\le M}$ is linearly independent. We denote by $\{\hat{\psi}_k^\phi\}_{1\leq k \leq M}$ the associated orthonormal family given by the Gram-Schmidt process. Notice that $\hat{\psi}_k^\phi = \psi_k^\phi$ for $k \in \{1,\dots,n\}$ and $\hat{\psi}_k^\phi \underset{\bo \phi \to \bo z}{\longrightarrow} \psi_k^z$ in~$H^3_{(V)}$  for $k=1,\dots,M$.
Let $\UU_\phi \in U(L^2)$ be any operator such that $\UU_\phi \psi_j^z = \hat{\psi}_j^\phi$  for every $j \in \{1,\dots,M\}$. By construction we have that $\UU_\phi \bo z = L_\phi \bo z = \bo \phi$ and $\|\UU_\phi \varphi_j - \varphi_j\|_{H^3_{(V)}} \underset{\bo \phi \to \bo z}{\longrightarrow} 0$ for any $j \in \{1,\dots,M\}$. This ends the proof of Lemma~\ref{lemme_propagateur_H3}.

\end{proof}

\subsection{Proof of the Main Theorem}

  Let us fix an arbitrary    $V\in H^4 $, and
  let   $\QQ_V $ be the set of all functions $\mu\in H^4 $ such that Conditions  $\mathrm{(C_6)}$ and $\mathrm{(C_7)}$   are satisfied with the  functions $V$ and $\mu$ replaced by the functions $V+\mu$ and $\mu$.  Let us prove that \eqref{syst_Neq} is exactly controllable in $\boldsymbol H^{4}_{(V)}$ for any $\mu\in \QQ_V$. 
   Along with      \eqref{syst_Neq}, let us
consider the system 
\begin{equation} \label{syst_Neq2}
\left\{
\begin{aligned}
& i \partial_t \psi^j = \big(- \partial^2_{xx} + V(x)+\mu(x)\big) \psi^j - u(t) \mu(x) \psi^j,  &(t,x)& \in (0,T) \times (0,1),
\\
& \psi^j(t,0) = \psi^j(t,1) = 0, \;  &j&   \in\{1, \dots, N\},
\\
& \psi^j(0,x) = \psi^j_0(x),
\end{aligned}
\right.
\end{equation}and denote by $\tilde \bp$   its resolving operator. Clearly, we have 
\begin{equation}\label{4.1}
\tilde \bp(t,   \bp_0, u)=\bp(t,   \bp_0, u-1)
\end{equation} for any $\bp_0\in \boldsymbol H^{3}_{(0)}$, $t\in [0,T]$, and $u\in L^2((0,T), \R)$. 
By Theorem~\ref{T:4.1}, system~\eqref{syst_Neq2} is exactly controllable in $\boldsymbol\SS\cap \boldsymbol H^{4}_{(V+\mu)} $ for   any $\mu\in \QQ_V$.

  Let us take any $\bp_0,\bp_f\in\boldsymbol \SS\cap \boldsymbol H^{4}_{(V)} $ and   any control $ u_1\in  W^{1,1}((0,1),\R)$   such that 
$u_1(0)=0$   and  $u_1(1)=-1$.  By Proposition~\ref{P:2.1},  $\bp(1,   \bp_0, u_1)=: \bp_{01}\in  \boldsymbol \SS\cap \boldsymbol H^{4}_{(V+\mu)}$ and $\bp(1,  \overline{ \bp_f}, u_1)=: \overline{\bp_{f1}}\in  \boldsymbol \SS\cap \boldsymbol H^{4}_{(V+\mu)}$. The time reversibility property implies that   $\bp(1,   { \bp_{f1}}, u_2)=  {\bp_{f}}$, where $u_2(t)=u_1(1-t),  t\in [0,1]$. Since \eqref{syst_Neq2} is  exactly controllable,    there is a time $\tilde T$ and a control $\tilde u\in L^2((0,T), \R)$ such that $\tilde \bp(\tilde T,   \bp_{01}, \tilde u)= \bp_{f1}$. Finally, choosing $T=\tilde T+ 2$ and   $u(t)=u_{1}(t)$ for $t\in (0,1)$, $u(t)=\tilde u(t-\tilde T)-1$ for $t\in (1,1+\tilde  T)$, and  $u(t)=u_{2}(t-1-\tilde{T})$ for $t\in (1+ \tilde T, T)$, we get  $
 \bp(T,\bp_0,u)= \bp_f 
 $. This proves the global exact controllability of \eqref{syst_Neq} in $\boldsymbol H^{4}_{(V)}$ for any $\mu\in \QQ_V$. 
 
 \smallskip
 
It remains to show that the set $\QQ_V$ is residual in $H^4 $. Let us write $\QQ_V=  \QQ_V^6 \cap \QQ_V^7$, where 
  $\QQ_V^j$ is the  set of all functions $\mu\in H^4 $ such that Condition  $\mathrm{(C_j)}$  is satisfied with $V$ and $\mu$ replaced by $V+\mu$ and $\mu$, $j=6,7$. Since  the intersection of two residual sets is residual, the proof of the Main Theorem follows from the following result. 
\begin{lemma}\label{L:4.1}
For any $V\in H^s, s\ge4$, the sets $\QQ_V^6$ and $\QQ_V^7$ are residual in~$H^s $. 
\end{lemma}
This lemma is proved in Section~\ref{GEN}. 
See~\cite{MasonSigalotti10} for the proof of the fact that~$\QQ_V^7$ is residual in a much more general case. Nevertheless, we give its proof in the Appendix, since it is     simpler in our setting.

\section{Appendix}

\subsection{Moment problem}
\label{annexe_pb_moment}

In this article, we use several times the following result about the   trigonometric moment problem.
\begin{proposition} \label{prop_pb_moment}
 Assume that Condition $\mathrm{(C_4)}$ is satisfied. Let $(\omega_n)_{n \in \N}$ be the increasing sequence defined by
\begin{equation*}
\{ \omega_n \, ;\, n \in \N \} = \{ \lambda_{k,V}-\lambda_{j,V} \, ; \, j \in \{1,\dots,N\}, \, k \geq j+1 \text{ and } k=j=N \}.
\end{equation*}
Then, for any $T>0$, there exists a continuous linear map
\begin{equation*}
\LL : \ell^2_r(\N,\C) \rightarrow L^2((0,T),\R)
\end{equation*}
such that for every $d=(d_n)_{n\in \N} \in \ell^2_r(\N,\C)$, we have 
\begin{equation*}
\int_0^T \LL(d)(t) e^{i \omega_n t} \md t = d_n, \quad \forall n \in \N.
\end{equation*}
\end{proposition}

\begin{proof} 
Let us set $\omega_{-n} := -\omega_n$ for $n \in \N$, and   let $D^+$ be the upper density of the sequence $(\omega_n)_{n\in \Z}$, i.e.,
\begin{equation*}
D^+ := \underset{r \to \infty}{\lim} \frac{n^+(r)}{r},
\end{equation*}
where $n^+(r)$ is the largest number of elements of the sequence $(\omega_n)_{n\in \Z}$ in an interval of length $r$. By the Beurling theorem (e.g., see~\cite[Theorem 9.2]{KomornikLoreti05}), if the uniform
gap condition    
\begin{equation} \label{gap_pb_moment}
\omega_{n+1} - \omega_n \geq \gamma, \quad \forall n \in \N
\end{equation} is satisfied for some $\gamma >0$,  
then for any for $T> 2\pi D^+$, the family $(e^{i \omega_n \cdot})_{n \in \Z}$ is a Riesz basis of $H_0 := \text{Adh}_{L^2(0,T)} \big( \text{Span} \{ e^{i \omega_n \cdot} \, ; \, n \in \Z \} \big)$. Let us show that, under Condition $\mathrm{(C_4)}$,  the  sequence $(\omega_n)_{n\in \Z}$ has a uniform
gap    and $D^+ =0$.  

\smallskip
\noindent
Indeed,  by the well-known asymptotic formula for the eigenvalues (e.g., see~\cite[Theorem 4]{PoschelTrubowitz}),  
\begin{equation} \label{estimee_valeur_propres}
\lambda_{k,V} = k^2 \pi^2 + \int_0^1 V(x) \md x + r_k, \quad \text{with } \sum_{k=1}^{\infty} r_k^2 < + \infty.
\end{equation}
This implies that for some sufficiently large integers   $n_0$ and $ k_0$, we have   
\begin{equation*}
\omega_{n_0+n} = \lambda_{k_0+p,V} - \lambda_{j,V}, \quad \text{where } n= p N + j, \: 1 \leq j \leq N, p\in \N.
\end{equation*}
Thus, the frequencies $(\omega_n)_{n\geq n_0}$ can be gathered as successive packets of $N$ frequencies such that  the minimal gap inside each packet is 
\begin{equation*}
\tilde{\gamma} :=  \min\limits_{1\leq q < m \leq N}( \lambda_{m,V}-\lambda_{q,V}).
\end{equation*}
Using Condition $\mathrm{(C_4)}$, we obtain $\tilde{\gamma} >0$.
The gap between the $(\ell+1)^{th}$ packet and the $\ell^{th}$ packet is 
$$\lambda_{\ell+1,V}-\lambda_{\ell,V}+ \lambda_{1,V}-\lambda_{N,V}$$which goes to infinity as $\ell\to\ty$, by \eqref{estimee_valeur_propres}.   On the other hand,  $\omega_n \neq \omega_k$ for $n \neq k$, by Condition $\mathrm{(C_4)}$. Hence   we get the uniform gap condition~(\ref{gap_pb_moment}).  From (\ref{estimee_valeur_propres}) it follows immediately that $D^+=0$.
Thus the family $(e^{i \omega_n \cdot})_{n \in \Z}$ is a Riesz basis of $H_0$. This implies that the map 
\begin{equation*}
\begin{array}{cccc}
J_0: &  H_0
& \rightarrow & \ell^2(\Z,\C)
\\
& f & \mapsto &  \left( \int_0^T f(t) e^{i \omega_n t} \md t \right)_{n \in \Z}
\end{array}
\end{equation*}
is an isomorphism. Then, the map $\LL : d \in \ell^2_r(\N,\C) \mapsto J_0^{-1}(\tilde{d})$, where $\tilde{d}_n := d_n$   and $\tilde{d}_{-n} := \overline{d_n}$ for $n \in \N$, satisfies the required properties.

\end{proof}

 
 \subsection{Proof of Lemma~\RefNeqsSecGenericite}\label{GEN}

\textit{First step}. Let us show that   $\QQ_V^7$ is residual in $H^s$. It suffices to show that the set  $\QQ_0^7$ of all functions $W\in H^s$, such that the numbers $\{1, \la_{j,W}\}_{j\in \N^*}$	 are rationally independent, is residual in  $H^s$. Let us take any 
 $M\in \N^*$ and $\boldsymbol r\in \Q^{M+1}\backslash\{\boldsymbol 0\}$ and denote by  $\QQ_{M,\boldsymbol r}$ the set of all functions $W\in H^s$ such that
 $$
r_0 + \sum_{j=1}^M r_j \la_{j,W}\neq 0.
 $$Then we have $\QQ_0^7 =\bigcap_{M\in \N^*, \boldsymbol r\in \Q^M\backslash\{\boldsymbol 0\}}\QQ_{M,\boldsymbol r}$. Thus it is sufficient to prove that $\QQ_{M,\boldsymbol r}$ is open and dense in $H^s $.
Continuity of the eigenvalues\footnote{By~\cite[Theorem~3]{PoschelTrubowitz}, the eigenvalues $\la_{k,W}$ and eigenfunctions $\varphi_{k,W}$ are real-analytic functions with respect to  $W\in L^2$.}  $\la_{k,W}$ from $L^2 $ to~$\R$ implies that $\QQ_{M,\boldsymbol r}$ is open   in $H^s $.
  Let us show that
$\QQ_{M,\boldsymbol r}$ is dense in $H^s$. For any $W, P\in H^s$ and $\sigma\in \R$, differentiating    the identity  
$$(- \partial^2_{xx}+W +\sigma P-\la_{j,  W+\sigma P})\varphi_{j,  W+ \sigma P}=0$$
with respect to $\sigma$ at $\sigma=0$, we get
$$
(-\partial^2_{xx}+W -\la_{j, W})\frac{\dd \varphi_{j,W+\sigma
 P} }{\dd
\sigma}\Big|_{\sigma=0}+(P-\frac{\dd \la_{j,W+\sigma P } }{\dd \sigma}\Big|_{\sigma=0})\varphi_{j,W }=0.
$$Taking the scalar product of this identity with $\varphi_{j,W}$, we
obtain
$$
\frac{\dd \la_{j,W+\sigma
 P} }{\dd
\sigma}\Big|_{\sigma=0}= \langle P,  \varphi_{j,W} ^2  \rangle.
$$
Thus
\begin{equation}\label{5.3}
\frac{\dd  }{\dd
\sigma}\Big( r_0 + \sum_{j=1}^M r_j \la_{j,W+\sigma
 P}\Big)\Big|_{\sigma=0}= \langle P,  \sum_{j=1}^M r_j \varphi_{j,W} ^2    \rangle.
\end{equation}
By~\cite[Theorem 9]{PoschelTrubowitz}, for any $W\in L^2 $,   the functions $\{\varphi_{j,W}^2\}_{j=1}^\ty$ are linearly independent. Hence we can find   $P\in H^s$ such that
$$
\langle P,  \sum_{j=1}^M r_j \varphi_{j,W} ^2  \rangle\neq0.
$$Then (\ref{5.3}) implies that $W+\sigma
 P\in\QQ_{M,\boldsymbol r}$ for any $\sigma$ sufficiently close to $0$. This shows that $\QQ_{M,\boldsymbol r}$ is dense in $H^s$. Thus  $\QQ_V^7$ is residual in $H^s$.

\medskip
\textit{Second step}. Recall that  $\QQ_V^6$ is the set of all functions $\mu \in H^s$ such that for any $j\in \N^*$ there exists $C_j > 0$   verifying
 $$|\lag \mu \varphi_{j,V+\mu}, \varphi_{k,V+\mu} \rag| \ge \frac{C_j }{k^{3}}\quad\text{for all $k \in \N^*$.}$$   
We will use the following   well known estimates for any $W\in L^2$
 \begin{align}
 & \|\varphi_{k,W}-\varphi_{k,0}\|_{L^\ty}\le \frac{C}{k},\label{5.4}\\
 &   \|  \varphi_{k,W}'-  \varphi_{k,0} '  \|_{L^\ty}\le C,\label{5.5}
 \end{align}   (e.g., see in~\cite[Theorem 4]{PoschelTrubowitz}).
 Integrating by parts, we get for any $W\in H^s$
 \begin{align*}
 \lag \mu \varphi_{j,W} , \varphi_{k,W}  \rag= & \frac{1}{\la_{k,W} } \lag  (-\partial^2_{xx}+W)(\mu \varphi_{j,W}), \varphi_{k,W} \rag\\=&\frac{1}{\la_{k,W}} (\lag -   \mu '' \varphi_{j,W}, \varphi_{k,W}\rag +2\lag  - \mu'  \varphi_{j,W}', \varphi_{k,W}\rag\\&\quad+\la_{j,W}\lag   \mu \varphi_{j,W}, \varphi_{k,W}\rag).
\end{align*}  This implies that for $k\neq j$, we have
\begin{align}\label{5.6}
 \lag \mu \varphi_{j,W} , \varphi_{k,W}  \rag&=    \frac{1}{\la_{j,W}-\la_{k,W}} (\lag    \mu '' \varphi_{j,W}, \varphi_{k,W}\rag +2 \lag   \mu'  \varphi_{j,W}', \varphi_{k,W}\rag) .
\end{align}
Again integrating by parts, we obtain
 \begin{align}\label{5.7}
 \lag   \mu'  \varphi_{j,W}', \varphi_{k,W}\rag = &\frac{1}{\la_{k,W}} \lag  \mu '     \varphi_{j,W}', (-\partial^2_{xx}+W) \varphi_{k,W}\rag\nonumber\\ = &-\frac{1}{\la_{k,W}}  \mu '     \varphi_{j,W}'   \varphi_{k,W}'\Big|_{x=0}^{x=1}+\frac{1}{\la_{k,W}} \lag (-\partial^2_{xx}+W) (  \mu '     \varphi_{j,W}', \varphi_{k,W}\rag.
\end{align}Using \eqref{5.6} with $\mu$ replaced by $\mu''$, we get
$$
 \lag \mu'' \varphi_{j,W} , \varphi_{k,W}  \rag=    \frac{1}{\la_{j,W}-\la_{k,W}} (\lag    \mu^{(4)} \varphi_{j,W}, \varphi_{k,W}\rag +2\lag   \mu^{(3)}  \varphi_{j,W}', \varphi_{k,W}\rag) .
$$ 
Combination of this last equality with \eqref{estimee_valeur_propres}, \eqref{5.4}-\eqref{5.7} and the explicit expression $\varphi_{k,0}(x)~=~\sqrt{2}\sin(k\pi  x)$, yields that 
 \begin{align*} 
k^3 \lag \mu \varphi_{j,W} , \varphi_{k,W}  \rag&=-4j\pi^{-1}\mu '       \cos(j\pi x)\cos(k\pi x)\Big|_{x=0}^{x=1}+  {c_{k,j}}{k^{-1}}\nonumber\\&=-4j\pi^{-1}((-1)^{j+k}\mu '(1)-\mu'(0))        +  {c_{k,j}}{k^{-1}} ,
\end{align*} where for any $j\in \N^*$ the sequence $c_{j,k}, k>j$ is bounded in $\R$.  Thus for any $\mu$ from the set
$$
\BB=\{\mu \in H^s \, ; \, \mu '(1) \pm \mu'(0) \neq 0\}
$$ and for any   $W\in H^s$,    there is $K_j\in \N^*$ such that  
 $$|\lag \mu \varphi_{j,W}, \varphi_{k,W} \rag| \ge \frac{C_j }{k^{3}}$$   
  for all $k \ge K_j$. In particular, this is true for $W=V+\mu$. Combining this with the following result, we complete the proof.
  \begin{lemma}\label{L:5.1}
  For any $V\in H^s$, the set $\QQ_V^1$   of all functions $\mu \in H^s$ such that  
  \begin{equation}\label{5.8}
  \lag \mu \varphi_{j,V+\mu}, \varphi_{k,V+\mu} \rag \neq 0
  \end{equation}   for all $j,k \in \N^*$, is residual in $H^s$.
 \end{lemma}
  Indeed, $\BB$ is open and dense in $H^s$ and  $\BB\cap \QQ_V^1\subset \QQ_V^6$. Then $\BB\cap \QQ_V^1$ is residual as an  intersection of two residual sets. Hence $\QQ_V^6$ is a residual set in~$H^s$.
  \begin{proof}[Proof of Lemma~\ref{L:5.1}.]
 For any $j,k\in \N^*$, let $\QQ_{V,j,k}^1$ be the set of  functions $\mu\in H^s$ such that \eqref{5.8} holds. Then $\QQ_V^1=\cap_{j,k\in \N^*} \QQ_{V,j,k}^1$ and it suffices to show that $\QQ_{V,j,k}^1$ is open and dense in $H^s$.  As above, the fact that $\QQ_{V,j,k}^1$ is open follows immediately from the continuous dependence of the eigenfunction~$\varphi_{k,V+\mu}$ on~$\mu$. Let us show that $\QQ_{V,j,k}^1$ is dense in $H^s$.
 Since $\varphi_{j,V}(x) \varphi_{k,V}(x) $ is not identically equal to zero, the set of    functions $\mu $ such that 
$\lag \mu \varphi_{j,V}, \varphi_{k,V} \rag \neq 0$  
is dense in $ H^s$. For any $\mu_0$ from that set,  the function  $\lag \mu_0 \varphi_{j,V+s\mu_0}, \varphi_{k,V+s\mu_0} \rag$ is non-zero    real-analytic  function with respect to $s\in \R$. Thus $s \mu_0\in \QQ_{V,j,k}^1$ almost surely for any $s\in \R$.    This   proves that $\QQ_{V,j,k}^1$ is dense in $H^s$.

\end{proof}

\section*{Conclusion and open problems}
In this article, we have proved simultaneous global exact controllability between any unitarily equivalent $N$-tuples of   functions in $ \SS \cap   H^4_{(V)}$. Our result is valid  in large time, for an arbitrary number of equations, and for an arbitrary potential. Hence, the spectrum of the free operator can be extremely resonant. Thus, not only we extend previous results on exact controllability for a single particle to  simultaneous controllability of $N$ particles, but we also improve the existing literature in \textsc{1d} for $N=1$.

Our proof combines several ideas. Using a Lyapunov strategy, we proved that any initial condition can be driven arbitrarily close to some finite sum of eigenfunctions. Then, designing a reference trajectory and using a rotation phenomenon on a suitable time interval we proved local exact controllability in $\bo H^3_{(V)}$ around $\bo \varphi$. Finally combining linearity of the equation with respect to the state and a compactness argument, we obtained global exact controllability under favourable hypotheses. The case of an arbitrary potential is dealt with a perturbation argument.

We mention here two possible ways to improve this result. The optimal functional setting for exact controllability is $H^3_{(V)}$. While using our Lyapunov function, we have dealt with more regular initial and final conditions to get convergence in $H^3$ from the boundedness in $H^4$. This issue of strong stabilization in infinite dimension is not specific to bilinear quantum system and is an open problem. The other possible improvement concerns the time of control. In our strategy, there are three steps requiring a time large enough : the approximate controllability, the rotation argument in local exact controllability, and the compactness argument.

\paragraph*{Acknowledgements.}  The first author thanks K.~Beauchard for having drawn his attention to the problem of simultaneous controllability and fruitful discussions. The authors were partially supported by ANR grant EMAQS No. ANR-2011-BS01-017-01, the second author was also partially supported by ANR grant STOSYMAP No. ANR-2011-BS01015-01.


\begin{thebibliography}{10}

\bibitem{BallMarsdenSlemrod82}
J.~M. Ball, J.E. Marsden, and M.~Slemrod.
\newblock Controllability for distributed bilinear systems.
\newblock {\em SIAM J. Control Optim.}, 20(4):575--597, 1982.

\bibitem{Beauchard05}
K.~Beauchard.
\newblock Local controllability of a 1-{D} {S}chr\"odinger equation.
\newblock {\em J. Math. Pures Appl. (9)}, 84(7):851--956, 2005.

\bibitem{BeauchardCoron06}
K.~Beauchard and J.-M. Coron.
\newblock Controllability of a quantum particle in a moving potential well.
\newblock {\em J. Funct. Anal.}, 232(2):328--389, 2006.

\bibitem{BeauchardLaurent}
K.~Beauchard and C.~Laurent.
\newblock Local controllability of 1{D} linear and nonlinear {S}chr\"odinger
  equations with bilinear control.
\newblock {\em J. Math. Pures Appl. (9)}, 94(5):520--554, 2010.

\bibitem{BeauchardMirrahimi09}
K.~Beauchard and M.~Mirrahimi.
\newblock Practical stabilization of a quantum particle in a one-dimensional
  infinite square potential well.
\newblock {\em SIAM J. Control Optim.}, 48(2):1179--1205, 2009.

\bibitem{BCCS11}
U.~Boscain, M.~Caponigro, T.~Chambrion, and M.~Sigalotti.
\newblock A weak spectral condition for the controllability of the bilinear
  {S}chr\"odinger equation with application to the control of a rotating planar
  molecule.
\newblock {\em Comm. Math. Phys.}, 311(2):423--455, 2012.

\bibitem{BoscainCaponigroSigalotti13}
U.~Boscain, M.~Caponigro, and M.~Sigalotti.
\newblock Multi-input {S}chr\"odinger equation: controllability, tracking, and
  application to the quantum angular momentum.
\newblock Preprint, arXiv:1302.4173, 2013.

\bibitem{BoscainChambrionSigalotti_review}
U.~Boscain, T.~Chambrion, and M.~Sigalotti.
\newblock On some open questions in bilinear quantum control.
\newblock Preprint, arXiv:1304.7181, 2013.

\bibitem{BCClogical_gate}
N.~Boussaid, M.~Caponigro, and T.~Chambrion.
\newblock {Implementation of logical gates on infinite dimensional quantum
  oscillators}.
\newblock In {\em {Proceedings of the American Control Conference 2012}}, pages
  5825--5830, Montreal, Canada, 2012.
\newblock 6 pages Programme INRIA Nancy Grand Est Color.

\bibitem{CMSB09}
T.~Chambrion, P.~Mason, M.~Sigalotti, and U.~Boscain.
\newblock Controllability of the discrete-spectrum {S}chr\"odinger equation
  driven by an external field.
\newblock {\em Ann. Inst. H. Poincar\'e Anal. Non Lin\'eaire}, 26(1):329--349,
  2009.

\bibitem{CoronBook}
J.-M. Coron.
\newblock {\em Control and nonlinearity}, volume 136 of {\em Mathematical
  Surveys and Monographs}.
\newblock American Mathematical Society, Providence, RI, 2007.

\bibitem{KomornikLoreti05}
V.~Komornik and P.~Loreti.
\newblock {\em Fourier series in control theory}.
\newblock Springer Monographs in Mathematics. Springer-Verlag, New York, 2005.

\bibitem{MasonSigalotti10}
P.~Mason and M.~Sigalotti.
\newblock Generic controllability properties for the bilinear {S}chr\"odinger
  equation.
\newblock {\em Comm. Partial Differential Equations}, 35(4):685--706, 2010.

\bibitem{Mirrahimi09}
M.~Mirrahimi.
\newblock Lyapunov control of a quantum particle in a decaying potential.
\newblock {\em Ann. Inst. H. Poincar\'e Anal. Non Lin\'eaire},
  26(5):1743--1765, 2009.

\bibitem{Morancey_polarisabilite}
M.~Morancey.
\newblock Explicit approximate controllability of the {S}chr\"odinger equation
  with a polarizability term.
\newblock {\em Math. Control Signals Systems}, pages 1--26, 2012.
\newblock Online First, DOI:10.1007/s00498-012-0102-2.

\bibitem{Morancey_simultane}
M.~Morancey.
\newblock Simultaneous local exact controllability of {1D} bilinear
  {S}chr\"odinger equations.
\newblock {\em Ann. Inst. H. Poincaré Anal. Non Linéaire}, 2013.
\newblock DOI : 10.1016/j.anihpc.2013.05.001.

\bibitem{Nersesyan}
V.~Nersesyan.
\newblock Growth of {S}obolev norms and controllability of the {S}chr\"odinger
  equation.
\newblock {\em Comm. Math. Phys.}, 290(1):371--387, 2009.

\bibitem{Nersesyan10}
V.~Nersesyan.
\newblock Global approximate controllability for {S}chr\"odinger equation in
  higher {S}obolev norms and applications.
\newblock {\em Ann. Inst. H. Poincar\'e Anal. Non Lin\'eaire}, 27(3):901--915,
  2010.

\bibitem{NersesyanNersisyan1D}
V.~Nersesyan and H.~Nersisyan.
\newblock Global exact controllability in infinite time of {S}chr\"odinger
  equation.
\newblock {\em J. Math. Pures Appl. (9)}, 97(4):295--317, 2012.

\bibitem{NersesyanNersisyan12}
V.~Nersesyan and H.~Nersisyan.
\newblock {Global exact controllability in infinite time of Schr{\"o}dinger
  equation: multidimensional case}.
\newblock Preprint, hal-00660478, 2012.

\bibitem{PoschelTrubowitz}
J.~P{\"o}schel and E.~Trubowitz.
\newblock {\em Inverse spectral theory}, volume 130 of {\em Pure and Applied
  Mathematics}.
\newblock Academic Press Inc., Boston, MA, 1987.

\bibitem{RouchonModele}
P.~Rouchon.
\newblock Control of a quantum particle in a moving potential well.
\newblock In {\em Lagrangian and {H}amiltonian methods for nonlinear control
  2003}, pages 287--290. IFAC, Laxenburg, 2003.

\bibitem{Schmidt}
W.M. Schmidt.
\newblock {\em Diophantine approximation}, volume 785 of {\em Lecture Notes in
  Mathematics}.
\newblock Springer, Berlin, 1980.

\bibitem{Turinici00}
G.~Turinici.
\newblock On the controllability of bilinear quantum systems.
\newblock In {\em Mathematical models and methods for ab initio quantum
  chemistry}, volume~74 of {\em Lecture Notes in Chem.}, pages 75--92.
  Springer, Berlin, 2000.

\bibitem{TuriniciRabitz04}
G.~Turinici and H.~Rabitz.
\newblock Optimally controlling the internal dynamics of a randomly oriented
  ensemble of molecules.
\newblock {\em Phys. Rev. A}, 70:063412, Dec 2004.

\end{thebibliography}
\end{document}